\documentclass[smallextended,envcountsame]{svjour3}       

\usepackage{amsfonts,amsmath,amssymb}
\usepackage{mathtools}
\usepackage{enumitem}
\usepackage{thm-restate}
\usepackage{wrapfig}
\usepackage{tikz}
\usepackage{booktabs}
\usepackage{multirow}
\usepackage{hyperref}
\usepackage[caption=false]{subfig}

\smartqed

\newcommand{\name}[1]{#1}
\newcommand{\tech}[1]{\textsf{#1}}

\newcommand{\bolda}{\mathbf{a}}
\newcommand{\boldb}{\mathbf{b}}
\newcommand{\boldc}{\mathbf{c}}
\newcommand{\bolde}{\mathbf{e}}
\newcommand{\boldq}{\mathbf{q}}
\newcommand{\boldx}{\mathbf{x}}
\newcommand{\boldy}{\mathbf{y}}
\newcommand{\boldalpha}{\boldsymbol{\alpha}}

\newcommand{\boldzero}{\boldsymbol{0}}
\newcommand{\imag}{\mathbf{i}}
\newcommand{\N}{\mathbb{N}}
\newcommand{\Z}{\mathbb{Z}}
\newcommand{\R}{\mathbb{R}}
\newcommand{\Q}{\mathbb{Q}}
\newcommand{\C}{\mathbb{C}}
\newcommand{\V}{\mathbb{V}}

\newcommand{\inner}[2]{\left\langle #1, #2 \right\rangle}

\DeclareMathOperator{\conv}{conv}
\DeclareMathOperator{\supp}{supp}
\DeclareMathOperator{\init}{init}
\DeclareMathOperator{\vvol}{vol}
\DeclareMathOperator{\mvol}{MV}

\newcommand{\revision}[1]{#1}
\def\makeheadbox{\relax}

\begin{document}

\title{Unmixing the mixed volume computation\thanks{%
    Research partially supported by an AMS-Simons Travel Grant 
    NSF Grant DMS 1115587, \revision{and the Auburn University Grant-In-Aid program}.
}}
\author{Tianran Chen}


\institute{Tianran Chen 
	\at Department of Mathematics and Computer Science,
	Auburn University at Montgomery,
	Montgomery Alabama USA
	\email{ti@nranchen.org}
}


\maketitle

\begin{abstract}
    Computing mixed volume of convex polytopes is an important problem 
    in computational algebraic geometry.
    This paper establishes sufficient conditions under which 
    the mixed volume of several convex polytopes exactly equals the 
    normalized volume of the convex hull of their union.
    Under these conditions the problem of computing mixed volume of several
    polytopes can be transformed into a volume computation problem for 
    a single polytope in the same dimension.
    We demonstrate through problems from real world applications 
    that substantial reduction in computational costs can be achieved via
    this transformation in situations where the convex hull of the union
    of the polytopes has less complex geometry than the original polytopes.
	We also discuss the important implications of this result in the 
	polyhedral homotopy	method for solving polynomial systems.
\end{abstract}

\section{Introduction} \label{sec:intro}

The concept of mixed volume \cite{minkowski_theorie_1911} arises naturally in 
the interplay between \name{Minkowski} sum and volume in the study of convex polytopes.
For convex polytopes $Q_1,\dots,Q_n \subset \R^n$, their \emph{mixed volume} 
is the coefficient of $\lambda_1 \cdots \lambda_n$ in
$\vvol_n(\lambda_1 Q_1 + \cdots + \lambda_n Q_n)$,
denoted $\mvol(Q_1,\dots,Q_n)$.
\name{D. Bernshtein} established a crucial connection 
between mixed volume and algebraic geometry: 
The number of isolated complex solutions with nonzero coordinate of a 
Laurent polynomial system is bounded above by the mixed volume of the 
\name{Newton} polytopes of the equations
\cite{bernshtein_number_1975,bernshtein_newton_1976,khovanskii_newton_1978,kushnirenko_newton_1976}.
This result has sparked the development of homotopy continuation methods 
\cite{huber_polyhedral_1995,verschelde_homotopies_1994}.
In particular, \name{B. Huber} and \name{B. Sturmfels} developed a new method 
for computing mixed volume via \emph{mixed subdivision} which produces an important 
by-product --- the \emph{polyhedral homotopy method} \cite{huber_polyhedral_1995} 
for solving polynomial systems.
Subsequently, mixed volume computation has became an important research problem 
in computational algebraic geometry
\cite{chen_mixed_2014,chen_mixed_2017,chen_homotopy_2015,emiris_efficient_1995,gao_mixed_2000,gao_mixed_2003,gao_algorithm_2005,jensen_computing_2008,lee_mixed_2011,li_finding_2001,malajovich_computing_2016,michiels_enumerating_1999,mizutani_demics:_2008,mizutani_dynamic_2007,verschelde_mixed-volume_1996}.

\begin{wrapfigure}[14]{r}{0.36\textwidth}
    \centering
    \includegraphics[width=0.36\textwidth]{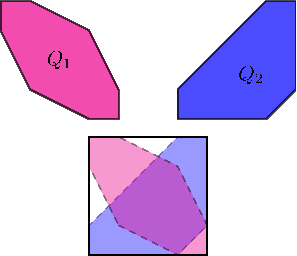}
    \caption{Two polygons and the convex hull of their union}
    \label{fig:example1}
\end{wrapfigure}
The computation of $\mvol(Q_1,\dots,Q_n)$ can be greatly simplified when 
some of the polytopes are identical, known as \emph{semi-mixed cases} 
\cite{gao_mixed_2003,huber_polyhedral_1995}.
The extreme case where all polytopes are identical, i.e., 
$Q_1 = \cdots = Q_n$, is known as the \textit{unmixed case}
which is equivalent to volume computation in the sense that 
$\mvol(Q,\dots,Q) = n! \vvol_n(Q)$.
Algorithms for calculating volume of polytopes can therefore be used in 
such unmixed cases.
The main goal of the present contribution is to establish conditions under which 
mixed volume computation can be turned into unmixed cases in the same dimension.
That is, we state conditions under which
$\mvol(Q_1,\dots,Q_n) = n! \vvol_n(\conv(Q_1 \cup \dots \cup Q_n))$
where ``$\conv$'' denotes \revision{the operation of taking} convex hull.

\begin{example}
    In $\R^2$, consider, the two convex polytopes (polygons)
    \begin{align*}
        Q_1 &= \conv \{ (1,1),(3,0),(4,0),(4,1),(3,3),(1,4),(0,4),(0,3) \} \\
        Q_2 &= \conv \{ (0,1),(0,0),(3,0),(4,1),(4,4),(3,4) \}
    \end{align*}
	in Fig. \ref{fig:example1}.
    We can verify that 
        $\mvol(Q_1,Q_2) = 2 \vvol_2(\conv(Q_1 \cup Q_2)) = 32$.
	The monotonicity of mixed volume implies that
	$\mvol(Q_1,Q_2) \le 2 \vvol_n(\conv(Q_1 \cup Q_2))$,
	but the equality will not hold in general, since one side is invariant
	under translations of $Q_1$ and $Q_2$ while the other is not.
	Understanding when and why this equality would hold is our main goal.
    Here, turning $\mvol(Q_1,Q_2)$ into $2\vvol(\conv(Q_1 \cup Q_2))$ 
    is computationally beneficial since $\conv(Q_1 \cup Q_2)$,
    having only 4 vertices, is significantly less complicated than 
    both $Q_1$ and $Q_2$.
\end{example}


We shall establish certain sufficient conditions under which
$\mvol(Q_1,\dots,Q_n)$ exactly equals 
$n! \vvol_n(\conv(Q_1 \cup \dots \cup Q_n))$.
They can be summarized into the following theorems
which clearly apply to the above example.
As we shall note in Remark \ref{rmk:automatic}, 
these conditions can, in principle, be checked automatically as by-products of
the process of computing the volume of $\conv(Q_1 \cup \dots \cup Q_n)$.

\begin{restatable}{theorem}{mainthma}
    \label{thm:mainthma}
    For finite sets $S_1,\dots,S_n \subset \Q^n$, 
    let $\tilde{S} = S_1 \cup \cdots \cup S_n$.
    If for every proper positive dimensional face $F$ of $\conv(\tilde{S})$ 
    we have $F \cap S_i \ne \varnothing$ for each $i = 1, \dots, n$ then
    $\mvol (\conv S_1, \dots, \conv S_n) = n! \vvol_n (\conv(\tilde{S}))$.
\end{restatable}


\begin{restatable}{theorem}{mainthmb}
    \label{thm:mainthmb}
    Given nonempty finite sets $S_1,\dots,S_n \subset \Q^n$, 
    let $\tilde{S} = S_1 \cup \cdots \cup S_n$.
    If every positive dimensional face $F$ of $\conv(\tilde{S})$ satisfies 
    one of the following conditions:
    \begin{description}[leftmargin=5ex]
        \item[(A)] 
            $F \cap S_i \ne \varnothing$ for all $i \in \{1,\dots,n\}$;
        
        \item[(B)] 
            $F \cap S_i$ is a singleton for some $i \in \{1,\dots,n\}$;
                    
        \item[(C)]
            For each $i \in I := \{ i \mid F \cap S_i \ne \varnothing \}$,
            $F \cap S_i$ is contained in a common coordinate subspace of
            dimension $|I|$, and the projection of $F$ to this subspace is
            of dimension less than $|I|$; 
    \end{description}
    then $\mvol (\conv(S_1), \dots, \conv(S_n)) = n! \vvol_n (\conv(\tilde{S}))$.
\end{restatable}

These theorems transform the mixed volume of $n$ polytopes into 
normalized volume of a single polytope in the same ambient space
--- the convex hull of their union.
Computationally, the potential advantage is three-fold:
\begin{enumerate}
    \item 
        When some of the $S_i$'s contain common points, such redundancy is 
        removed in the union $\tilde{S} := S_1 \cup \cdots \cup S_n$ 
        in the sense that $|\tilde{S}| < |S_1| + \cdots + |S_n|$.
        Since the number of vertices plays an important role in the complexity 
        of algorithms for manipulating polytopes,
        it may be much more efficient to study $\conv(\tilde{S})$ in such cases.

    \item
        In forming the union $\tilde{S} := S_1 \cup \cdots \cup S_n$,
        certain vertices of some of the $\conv(S_i)$ may no longer be vertices
        of $\conv(\tilde{S})$ and hence can be ignored in computing
        $n! \vvol(\conv(\tilde{S}))$.
        
    \item
        Currently, there appear to be a greater variety of efficient 
        algorithms for volume computation than for mixed volume computation
        (see \cite{bueler_exact_2000} and \cite{chen_homotopy_2015}).
\end{enumerate}
The combined effect of these computational advantages can lead to substantial
reduction of computational costs for certain problems 
as we shall demonstrate. 
\revision{Moreover, the equivalence of mixed volume and normalized volume
is likely to lead to alternative algorithms for approximating mixed volume
since there are well-studied polynomial time algorithms
for approximating volume in general.}

\revision{Of course, this transformation can also be used in reverse:
The volume of a single polytope could be reduced to mixed volume of
several simpler polytopes which can potentially be easier to compute. 
Indeed, this idea was used in the author's recent work~\cite{ChenDavis2017Product}
(with Robert Davis) to show that the normalized volume of a free sum of two polytopes
is simply the product of their normalized volume.
}

\revision{Interestingly, this problem is also studied by Fr\'ed\'eric Bihan 
and Ivan Soprunov around the same time~\cite{BihanSoprunov2017}
from a different point of view.
In particular, Theorem~\ref{thm:mainthma} turns out to be a special case
of their more general results.}

This paper is structured as follows: 
\revision{
    \S \ref{sec:motivation} shows a geometric the connection
    between mixed volume and volume via a simple example.
}
\S \ref{sec:notations} reviews concepts to be used.
The main proofs are given in \S \ref{sec:proofs}.
They are generalized to the semi-mixed cases in \S \ref{sec:semimix}.
Their applicability are demonstrated in \S \ref{sec:cases} via problems
from real-world applications.
\S \ref{sec:results} highlights the potentially substantial reduction in 
computational costs achieved by these transformations.
Finally, \S \ref{sec:homotopy} explores the implication in the 
polyhedral homotopy method for solving polynomial systems.

\section{\revision{A geometric motivation}}\label{sec:motivation}
\revision{
Before presenting the proofs, we use a simple example to show
intuitively why there could be any connections between mixed volume
of several polytopes and the normalized volume of the convex hull of their union at all.

Consider the simple 2-dimensional example shown in Figure~\ref{fig:simple-mcells} with 
\begin{align*}
    Q_1 &= \{ (0,0), (0,2), (2,0), (2,2) \} \\
    Q_2 &= \{ (0,0), (1,2), (2,1) \}
\end{align*}
which satisfies the assumptions in Theorem~\ref{thm:mainthma}:
all edges of $\conv(Q_1 \cup Q_2)$ intersect with both $Q_1$ and $Q_2$.
In Figure~\ref{fig:simple-mcells}, we can easily see that the 
Minkowski sum $Q_1 + Q_2$ contains a copy of $Q_1$ and a copy of $Q_2$.
Under the scaling of $Q_1 \mapsto \lambda_1 Q_1$ and $Q_2 \mapsto \lambda_2 Q_2$,
the area of those copies of $Q_1$ and $Q_2$ are scaled by 
$\lambda_1^2$ and $\lambda_2^2$ respectively.
The remaining regions whose area will scale with the factor $\lambda_1 \cdot \lambda_2$
are known as \emph{mixed cells}.
The sum of the areas of these mixed cells is exactly the mixed volume~\cite{huber_polyhedral_1995}.
In this case, the mixed volume is 8.
\begin{figure}[t]
    \centering
    \includegraphics[width=0.5\textwidth]{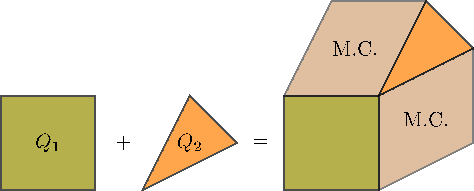}
    \caption{$Q_1$, $Q_2$, the Minkowski sum $Q_1 + Q_2$, and the mixed cells (M.C.) inside it}
    \label{fig:simple-mcells}
\end{figure}

We now examine the growth rate of the normalized volume
$v(\lambda) = 2 \vvol_2(Q_1 \cup \lambda Q_2)$ 
as a function of the positive real scalar $\lambda$.
From Figure~\ref{fig:simple-grow}, 
we can see it is a piecewise function:
For $\lambda \le 1$, $\lambda Q_2$ is contained in $Q_1$,
therefore $v(\lambda)$ remains a constant;
For $1 \le \lambda \le \frac{4}{3}$, vertices of $\lambda Q_2$
start to push out of $Q_1$ while edges remain partially in $Q_1$,
and $v(\lambda)$ grows linearly;
Finally, for $\lambda > \frac{4}{3}$,
an edge of $\lambda Q_2$ leaves $Q_1$,
and $v(\lambda)$ grows quadratically.

\begin{figure}[h]
    \centering
    \subfloat[For $\lambda < 1$, $v(\lambda)$ remains constant]{
        \includegraphics[width=0.155\textwidth]{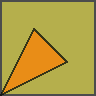}
    }\hspace{5ex}%
    \subfloat[For $1 \le \lambda \le 4/3$, $v(\lambda)$ grows linearly]{
        \includegraphics[width=0.185\textwidth]{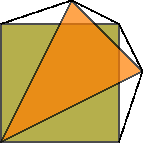}
    }\hspace{5ex}
    \subfloat[For $\lambda = 4/3$]{
        \includegraphics[width=0.205\textwidth]{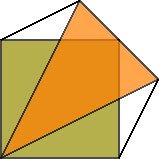}
    }\hspace{5ex}
    \subfloat[For $\lambda > 4/3$, $v(\lambda)$ grows quadratically]{
        \includegraphics[width=0.21\textwidth]{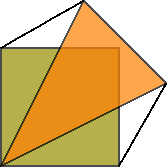}
    }
    \caption{
        The growth rate of $v(\lambda) = 2 \vvol_2(Q_1 \cup \lambda Q_2)$
        for different ranges of $\lambda$.
    }
    \label{fig:simple-grow}
\end{figure}

Indeed, it is easy to verify that
\begin{equation*}
    v(\lambda) = 
    2 \vvol_2 (\conv{Q_1 \cup \lambda Q_2}) =
    \begin{cases}
        8 & \text{if } \lambda \le 1 \\
        8 \lambda & \text{if } 1 < \lambda \le \frac{4}{3} \\
        4 \lambda + 3 \lambda^2 & \text{if } \lambda > \frac{4}{3}.
    \end{cases}
\end{equation*}
That is, for any $\lambda \in [1,4/3]$,
\begin{equation*}
    v(\lambda) = 2 \vvol_2(Q_1 \cup \lambda Q_2) = \mvol(Q_1,\lambda Q_2).
\end{equation*}
This interval $[1,4/3]$ for $\lambda$ is precisely the interval for which
the pair $(Q_1,\lambda Q_2)$ satisfies the conditions in Theorem~\ref{thm:mainthma},
i.e., all edges of $\conv(Q_1 \cup Q_2)$ intersect both $Q_1$ and $Q_2$.
This example also shows that the conditions required by Theorem~\ref{thm:mainthma}
are not just very special configurations of two polytopes
but can remain valid under a range of scaling.

A more direct connection between the normalized volume and mixed volume
can be visualized through mixed cells (Figure~\ref{fig:simple-mcells}).
Recall that the sum of the area of the mixed cells is precisely the mixed volume.
As shown in Figure~\ref{fig:simple-decomp}, for $\lambda \in [1,4/3]$,
e.g. $\lambda=1.2$ (Figure~\ref{fig:simple-grow}(c))
two copies of $\conv(Q_1 \cup \lambda Q_2)$ can be subdivided and rearranged
to form the two mixed cells of $(Q_1,\lambda Q_2)$ in Figure~\ref{fig:simple-mcells}.
Therefore, the normalized volume $2 \vvol_2(\conv(Q_1 \cup \lambda Q_2))$
is precisely the mixed volume $\mvol(Q_1,\lambda Q_2)$.

\begin{figure}[h]
    \centering
    \includegraphics[width=0.42\textwidth]{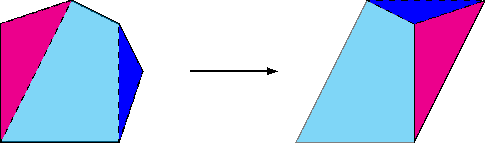}
    \hspace{8ex}%
    \includegraphics[width=0.42\textwidth]{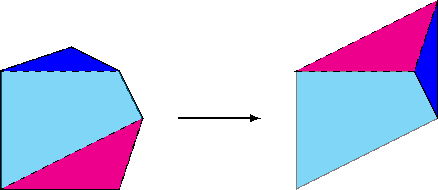}
    \caption{
        Two copies of $\conv(Q_1 \cup \lambda Q_2)$
        rearranged to form the
        mixed cells for $(Q_1,\lambda Q_2)$
    }
    \label{fig:simple-decomp}
\end{figure}

The geometric observation illustrated above shows a potential
connection between the mixed volume of polytopes and the 
volume of the convex hull of their union
under the assumptions of Theorem~\ref{thm:mainthma},
and this is the main motivation behind this study.
However, the author was unable to generalize this to a proof.
Using the theory of Bernshtein-Kushnirenko-Khovanskii bound for
Laurent polynomial systems, this paper presents a purely algebraic proof.
}

\section{Preliminaries} \label{sec:notations}

For a compact set $Q \subset \R^n$, $\vvol_n(Q)$ denotes its standard 
Euclidean volume, and the quantity $n! \vvol_n(Q)$ is known as its 
\emph{normalized volume}.
Given a nonzero vector $\boldalpha \in \R^n$, we also define
$h_{\boldalpha}(Q) := \min \{ \inner{\boldq}{\boldalpha} \mid \boldq \in Q \}$
and
$(Q)_{\boldalpha} := \{ \boldq \in Q \mid \inner{\boldq}{\boldalpha} = h_{\boldalpha}(Q) \}$.
A \emph{convex polytope} $P \subset \R^n$ is the convex hull of a finite set.
A subset of the form $(P)_{\boldalpha}$ is called a \emph{face} of $P$. 
Each face has a well defined \emph{dimension}.
A face of dimension 0 must be a point and is known as a \emph{vertex}.
Other nonempty faces are said to be \emph{positive dimensional}.


The \emph{Minkowski sum} of two sets $A,B \subset \R^n$ 
is $A + B = \{ \bolda + \boldb \mid \bolda \in A, \boldb \in B \}$.
Given convex polytopes $Q_1,\dots,Q_n \subset \R^n$, the Minkowski sum 
$\lambda_1 Q_1 + \cdots + \lambda_n Q_n$ is also a convex polytope, 
and its volume is a homogeneous polynomial in the positive scalars
$\lambda_1,\dots,\lambda_n$.
In it, the coefficient\footnote{%
    An alternative definition for mixed volume is the coefficient of
    $\lambda_1 \cdots \lambda_n$ in that polynomial divided by $n!$.
}    
of $\lambda_1 \cdots \lambda_n$ is the \emph{mixed volume} \cite{minkowski_theorie_1911} 
of $Q_1,\dots,Q_n$, denoted $\mvol(Q_1,\dots,Q_n)$.
Mixed volume is symmetric, additive, and nondecreasing
(see \S \ref{sec:mixvol}). 


Though the main results to be established are in the realm of geometry, 
our proofs take a decidedly algebraic approach
via the theory of root counting.
A \emph{Laurent monomial} in $\boldx=(x_1,\dots,x_n)$ induced by vector 
$\bolda=(a_1,\dots,a_n) \in \Z^n$ is the formal expression 
$\boldx^{\bolda} = x_1^{a_1} \cdots x_n^{a_n}$.
A \emph{Laurent polynomial} a linear combination of distinct Laurent monomials,
i.e., an expression of the form $p = \sum_{\bolda \in S} c_{\bolda} \boldx^{\bolda}$.
The set $S \subset \Z^n$ collecting all the exponent vectors is known as the 
\emph{support} of $p$, denoted $\supp(p)$,
\revision{and the convex hull $\conv(S)$ is known as the \emph{Newton polytope} of $p$.}
For a Laurent polynomial system $P = (p_1,\dots,p_m)$ in $\boldx = x_1,\dots,x_n$, 
we define $\V^*(P) := \{ \boldx \in (\C^*)^n \mid P(\boldx) = \boldzero \}$,
and it consists of
\emph{components} each with a well defined dimension.
Of special interest in our discussion are the components of
zero dimension which are the isolated points\footnote{%
    Here, a point $\boldx \in \V^*(P)$ is said to be \emph{isolated} 
    (a.k.a. geometrically isolated) if there is an open set in $(\C^*)^n$
    that contains $\boldx$ but does not contain any other points in $\V^*(P)$.
} in $\V^*(P)$.
This subset will be denoted by $\V_0^*(P)$.
The proofs of our main results rely on the following important theorems.

\begin{theorem}[\name{Kushnirenko} \cite{kushnirenko_newton_1976}]
    \label{thm:kushnirenko}
    For a Laurent polynomial system $P = (p_1,\dots,p_n)$ in 
    $\boldx = (x_1,\dots,x_n)$ with identical support,
    $S = \supp(p_i)$ for $i=1,\dots,n$, we have
    $|\V_0^*(P)| \le n! \vvol_n (\conv(S))$.
\end{theorem}


\begin{theorem}[\name{Bernshtein}'s First Theorem \cite{bernshtein_number_1975}]
    \label{thm:bernshtein-a}
    For a Laurent polynomial system $P = (p_1,\dots,p_n)$ 
    in $\boldx = (x_1,\dots,x_n)$ (with potentially different supports),
    $|\V_0^*(P)| \le \mvol (\conv(\supp(p_1)),\dots,\conv(\supp(p_n)))$.
\end{theorem}

In \cite{canny_optimal_1991}, this upper bound was nicknamed the \emph{BKK bound} 
after the works of Bernshtein \cite{bernshtein_number_1975,bernshtein_newton_1976}, 
Kushnirenko \cite{kushnirenko_newton_1976}, and Khovanskii \cite{khovanskii_newton_1978}.
The condition under which the BKK bound is exact (counting multiplicity)
is stated in terms of ``initial systems'':
For a Laurent polynomial $p = \sum_{\bolda \in S} c_{\bolda} \boldx^{\bolda}$ 
in $\boldx = (x_1,\dots,x_n)$ and a nonzero vector $\boldalpha \in \R^n$, 
$\init_{\boldalpha} (p) := \sum_{\bolda \in (S)_{\boldalpha}} c_{\bolda} \boldx^{\bolda}$.
For a Laurent polynomial system $P = (p_1,\dots,p_m)$,
the \emph{initial system} of $P$ with respect to $\boldalpha$ is 
$\init_{\boldalpha} (P) := 
( \init_{\boldalpha}(p_1), \dots, \init_{\boldalpha}(p_m) )$.

\begin{theorem}[\name{Bernshtein}'s Second Theorem \cite{bernshtein_number_1975}]
    \label{thm:bernshtein-b}
    For a Laurent polynomial system $P = (p_1,\dots,p_n)$ in 
    $\boldx = (x_1,\dots,x_n)$ with supports $S_1=\supp(p_1),\dots,S_n=\supp(p_n)$, 
    if for all $\boldalpha \in \R^n \setminus \{\boldzero\}$, 
    $\init_{\boldalpha}(P)$ has no zeros in $(\C^*)^n$,
    then all zeros of $P$ in $(\C^*)^n$ are isolated,
    and, counting multiplicity, the total number of zeros is exactly
    $\mvol (\conv(S_1),\dots,\conv(S_n))$.
\end{theorem}

\begin{remark}
    \label{rmk:generic}
	An important fact is that the BKK bound is always attainable.
	That is, fixing the supports of $P$, 
	there is always some choice of the coefficients for which all points 
	in $\V_0^*(P)$ are simple (of multiplicity one), 
	and $|\V_0^*(P)| = \mvol (\conv(\supp(p_1)),\dots,\conv(\supp(p_n)))$.
    In this case, $P$ is said to be in \emph{general position}
    (with respect to the supports).
	Indeed, such choices form a nonempty \emph{Zariski open set} in the coefficient space
	\cite{canny_optimal_1991,gelfand_discriminants_1994,huber_solving_1996,li_cheaters_1989,morgan_coefficient-parameter_1989,mumford_algebraic_1995}.
\end{remark}

\section{Proofs of the main results} \label{sec:proofs}

\revision{
    The proofs of the main theorems all rely on the theory of BKK bound.
    Given a system of $n$ Laurent polynomial systems
    $P = (p_1,\dots,p_n)$ in $n$ variables with supports $S_1,\dots,S_n$
    and Newton polytopes $Q_i = \conv(S_i)$.
    If $P$ is in general position (Remark~\ref{rmk:generic}) 
    the $\C^*$-root count is exactly the mixed volume $\mvol(Q_1,\dots,Q_n)$.
    For a nonsingular square matrix $A$,
    the systems $A \cdot P$ and $P$ have the exact same set of $\C^*$-roots.
    Moreover, if $A$ is chosen generically,
    $\supp(A \cdot P)$ is exactly the union of $S_1,\dots,S_n$.
    Therefore, we have 
    \[
        \mvol(Q_1,\dots,Q_n) = 
        |\V_0^*(P)| =
        |\V_0^*(A \cdot P)| \le
        n! \vvol(\conv(S_1 \cup \dots \cup S_n)).
    \]
    In general, we may not have the equality because $A \cdot P$
    itself may not be in general position as a member of the much larger family
    of Laurent polynomial systems with Newton polytopes $\conv(S_1 \cup \dots \cup S_n)$.
    We establish the equality by showing under certain conditions,
    if $P$ is general position, then $A \cdot P$ is also in general position.
}

\begin{wrapfigure}[12]{r}{0.30\textwidth}
    \centering
    \includegraphics[width=0.30\textwidth]{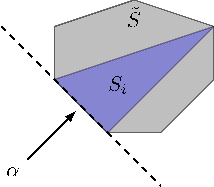}
    \caption{
        A face of $\conv(\tilde{S})$ intersecting $S_i$
        for certain $i$.
    }
    \label{fig:hyperplane}
\end{wrapfigure}
The proofs make repeated use of a simple geometric observation illustrated in 
Fig. \ref{fig:hyperplane}.
We state it as a lemma for later reference:

\begin{lemma}
    \label{lem:face}
    For $n$ nonempty sets $S_1,\dots,S_n \subset \Q^n$, 
    let $F$ be a proper face of $\conv(S_1 \cup \cdots \cup S_n)$
    and $\boldalpha$ be its inner normal.
    For each $i$ such that $F \cap S_i \ne \varnothing$,
    $F \cap S_i = (S_i)_{\boldalpha}$.
\end{lemma}

\begin{proof}
    Fix an $i \in \{1,\dots,n\}$ such that $F \cap S_i \ne \varnothing$.
    If we let $h = h_{\boldalpha} (\conv(S_1 \cup \cdots \cup S_n))$ then
    for any $\boldx \in F \cap S_i$, $\inner{\boldx}{\boldalpha} = h$
    while for any $\boldy \in S_i \setminus F \subseteq 
    \conv(S_1 \cup \cdots \cup S_n) \setminus F$
    we have $\inner{\boldy}{\boldalpha} > h$.
    Therefore $F \cap S_i = (S_i)_{\boldalpha}$.
    \qed
\end{proof}

With this lemma, we restate and prove the two main theorems listed earlier.

\mainthma*

\begin{proof}
    We shall first prove the cases where $S_i \subset \Z^n$
    for $i=1,\dots,n$.
    Let $P = (p_1,\dots,p_n)$ be a Laurent polynomial system in 
    $\boldx = (x_1,\dots,x_n)$ such that $\supp(p_i) = S_i$ for each $i$.
    That is, $p_i(\boldx) = \sum_{\bolda \in S_i} c_{i,\bolda} \boldx^{\bolda}$.
    By Theorem \ref{thm:bernshtein-a}, Theorem \ref{thm:bernshtein-b}, and Remark \ref{rmk:generic},
    there exists a choice of the coefficients 
    $\{c_{i,\bolda} \mid \bolda \in S_i, i=1,\dots,n\}$
    such that 
    \revision{$P$ is in general position, i.e.,}
    all points in $\V_0^*(P)$ are simple and
    \begin{equation}
        \label{equ:P-mv}
        |\V_0^*(P)| = \mvol (\, \conv (S_1), \dots, \conv (S_n) \,).
    \end{equation}
    Now, consider a \emph{randomization} \cite{sommese_numerical_2005}
    $A \cdot P$ of $P$ induced by a nonsingular $n \times n$ complex matrix $A = [a_{ij}]$:
    \begin{equation}
        \label{equ:AP}
        A \cdot P := 
        \begin{bmatrix}
            a_{11} & \cdots & a_{1n} \\
            \vdots & \ddots & \vdots \\
            a_{n1} & \cdots & a_{nn}
        \end{bmatrix}
        \begin{bmatrix}
            p_1 \\ \vdots \\ p_n
        \end{bmatrix} =
        \begin{bmatrix}
            a_{11} p_1 + \cdots + a_{1n} p_n \\
            \vdots \\
            a_{n1} p_1 + \cdots + a_{nn} p_n
        \end{bmatrix}.
    \end{equation}
    Since $A$ is nonsingular,  
    $(A \cdot P)(\boldx) = A \cdot (P(\boldx)) = \boldzero$
    if and only if $P(\boldx) = \boldzero$.
    Therefore, $A \cdot P$ and $P$ have the same zeros.
    In particular, $\V_0^*(A \cdot P) = \V_0^*(P)$, and all of its points are simple.
    
    With the coefficients of $P$ already fixed, we assume entries 
    of $A$ are chosen so that there are no cancellations of terms in $A \cdot P$, 
    then it is easy to verify that the supports of the Laurent polynomials 
    in $A \cdot P$ are identical which is
    \[
        \supp (a_{i1} p_1 + \cdots + a_{in} p_n) = S_1 \cup \cdots \cup S_n =: \tilde{S}
    \]
    for each $i=1,\dots,n$.
    By \name{Kushnirenko}'s Theorem (Theorem \ref{thm:kushnirenko}), 
    \begin{equation}
        \label{equ:AP-mv}
        |\V_0^*(A \cdot P)| \le
        n! \vvol_n (\conv(\tilde{S})).
    \end{equation}
    Theorem \ref{thm:bernshtein-b} and Remark \ref{rmk:generic} states that 
    the equality holds as long as for any nonzero vector $\boldalpha \in \R^n$, 
    the initial system 
    $\init_{\boldalpha} (A \cdot P)$ has no zero in $(\C^*)^n$.
    Let $F := (\conv(\tilde{S}))_{\boldalpha}$.
    If $F$ is a vertex, i.e., $F = \{\bolda\}$ for some $\bolda \in \tilde{S}$,
    then each component of $\init_{\boldalpha} (A \cdot P)$ has only one term: 
    the term involving $\boldx^{\bolda}$.
    Therefore $\init_{\boldalpha} (A \cdot P)$ has no zero in $(\C^*)^n$.
    Otherwise, $F$ is positive dimensional and hence, by assumption, 
    $F$ intersects each $S_i$.
    By Lemma \ref{lem:face}, $F \cap S_i = (S_i)_{\boldalpha}$ for each 
    $i=1,\dots,n$. Therefore,
    \begin{equation}
        \init_{\boldalpha} (A \cdot P) = 
        \begin{bmatrix}
        \displaystyle
        \sum_{i=1}^n a_{1i} 
        \sum_{\bolda \in F \cap S_i} c_{i,\bolda} \boldx^{\bolda}
        \\
        \vdots \\
        \displaystyle
        \sum_{i=1}^n a_{ni} 
        \sum_{\bolda \in F \cap S_i} c_{i,\bolda} \boldx^{\bolda}
        \\
        \end{bmatrix} =
        A \cdot 
        \begin{bmatrix}
        \displaystyle
        \sum_{\bolda \in F \cap S_1} c_{1,\bolda} \boldx^{\bolda} \\
        \vdots \\
        \displaystyle
        \sum_{\bolda \in F \cap S_n} c_{n,\bolda} \boldx^{\bolda} \\
        \end{bmatrix}
        =
        A \cdot 
        \begin{bmatrix}
        \init_{\boldalpha} (p_1) \\
        \vdots \\
        \init_{\boldalpha} (p_n)
        \end{bmatrix}.
    \end{equation}
    That is, $\init_{\boldalpha} (A \cdot P) = A \cdot \init_{\boldalpha}(P)$.
    Recall that $A$ is nonsingular, so $\init_{\boldalpha}(A \cdot P)=\boldzero$ 
    if and only if $\init_{\boldalpha}(P) = \boldzero$.
    But $P$ is in general position, so 
    \[
        \V^*(\init_{\boldalpha}(A \cdot P)) = 
        \V^*(\init_{\boldalpha}(P)) = \varnothing.
    \]
    Therefore for all nonzero vector $\boldalpha \in \R^n$, 
    $\init_{\boldalpha} (A \cdot P)$ has no zero in $(\C^*)^n$.
    Recall that all points in $\V_0^*(A \cdot P)$ are simple,
    so by \name{Bernshtein}'s Second Theorem (Theorem \ref{thm:bernshtein-b}),
    $|\V_0^*(A \cdot P)| = n! \vvol(\conv(\tilde{S}))$, and consequently
    \begin{equation*}
        \mvol(\conv S_1, \dots, \conv S_n) = 
        |\V_0^*(P)| = 
        |\V_0^*(A \cdot P)| = 
        n! \vvol(\conv(\tilde{S})).
    \end{equation*}
    Since both sides of this equality are homogeneous of degree $n$ in a uniform 
    positive scaling, this result directly extend to cases with $S_i \subset \Q^n$.
    \qed
\end{proof}

\mainthmb*

\begin{proof}
    
    We shall reuse the previous constructions:
    Let the Laurent polynomial system $P$, nonsingular matrix $A$, and the
    randomization $A \cdot P$ be those defined in the previous proof.
    Then $\supp (A \cdot P) = (\tilde{S},\dots,\tilde{S})$ and
    \[
        \mvol(\conv(S_1),\dots,\conv(S_n)) 
        = |\V^*_0(P)|
        = |\V^*_0(A \cdot P)| 
        \le n! \vvol (\conv(\tilde{S})).
    \]
    The goal is still to establish the equality by showing for any nonzero
    vector $\boldalpha \in \R^n$, $\init_{\boldalpha}(A \cdot P)$ has no zero
    in $(\C^*)^n$ under the above assumptions.
    Let $F = (\conv(\tilde{S}))_{\boldalpha}$.
    If $F$ is a vertex or satisfies condition (A), the proof for
    Theorem \ref{thm:mainthma} has already shown that
    $\V^*(\init_{\boldalpha}(A \cdot P)) = \varnothing$.
    
    To study the remaining possibilities, 
    we may assume $F$ is positive dimensional and $F \cap S_j = \varnothing$
    for some $j$.
    Let $I := \{ i \in \{1,\dots,n\} \mid F \cap S_i \ne \varnothing \}
    = \{i_1,\dots,i_m\}$ with $m = |I| < n$.
    Then by Lemma \ref{lem:face}, $F \cap S_i = (S_i)_{\boldalpha}$ for each
    $i \in I$. Hence
    \begin{align*}
        \init_{\boldalpha}(A \cdot P) =
        \begin{bmatrix}
            \displaystyle
            \sum_{i \in I} a_{1i} 
            \sum_{\bolda \in F \cap S_i} c_{i,\bolda} \boldx^{\bolda} 
            \\
            \vdots \\
            \displaystyle
            \sum_{i \in I} a_{ni} 
            \sum_{\bolda \in F \cap S_i} c_{i,\bolda} \boldx^{\bolda} 
        \end{bmatrix} 
        =
        \begin{bmatrix}
            \displaystyle
            \sum_{i \in I} a_{1i} \init_{\boldalpha}(p_i) \\
            \vdots \\
            \displaystyle
            \sum_{i \in I} a_{ni} \init_{\boldalpha}(p_i)
        \end{bmatrix}
        =
        A_I \cdot \init_{\boldalpha}(P_I)
    \end{align*}
    where $A_I$ is the matrix whose columns consists of the $i$-th columns of 
    $A$ for $i \in I$ and $P_I$ is the Laurent polynomial system 
    (as a column vector) whose components are $p_i$ for $i \in I$.
    Since $A$ is assumed to be nonsingular, $A_I$ has rank $|I|$.
    Therefore $\boldzero = \init_{\boldalpha}(A \cdot P) = A_I \cdot \init_{\boldalpha}(P_I)$
    if and only if $\init_{\boldalpha}(P_I) = \boldzero$.
    \smallskip
    
    \noindent Case (B) \hspace{1ex}
    Assume $F$ satisfies condition (B), then there is at least one $i_1 \in I$ 
    for which $F \cap S_{i_1} = (S_{i_1})_{\boldalpha}$ is a singleton, i.e., 
    $F \cap S_{i_1} = \{\bolda\}$ for some $\bolda \in S_{i_1}$.
    Therefore, $\init_{\boldalpha}(p_{i_1}) = c_{i_1,\bolda} \boldx^{\bolda}$
    where $c_{i_1,\bolda} \ne 0$, and hence it has no zero in $(\C^*)^n$.
    Consequently,
    \[
        \V^*(\init_{\boldalpha} (A \cdot P)) = 
        \V^*(A_I \cdot \init_{\boldalpha} (P_I)) =
        \V^*(\init_{\boldalpha} (P_I)) \subseteq
        \V^*(\init_{\boldalpha} (p_{i_1})) = \varnothing.
    \]
    That is, $\V^*(\init_{\boldalpha} (A \cdot P)) = \varnothing$.
    \smallskip
    
    \noindent Case (C) \hspace{1ex}
    Finally, assume $F$ satisfies condition (C).
    Then $\supp(\init_{\boldalpha}(P_I)) = (F \cap S_{i_1}, \dots, F \cap S_{i_m})$
    lie in a common coordinate subspace of dimension $m = |I|$.
    Let $j_1,\dots,j_m$ be the indices so that $\bolde_{j_1},\dots,\bolde_{j_m}$
    form a basis for this coordinate subspace.
    Then $\init_{\boldalpha} (P_I)$ only involves $m$ of the $n$ variables
    $x_{j_1},\dots,x_{j_m}$. 
    To emphasize this, we shall write it as $P_F = P_F(x_{j_1},\dots,x_{j_m})$, 
    and it is a square system of $m$ equations in $m$ variables.
    Let $\pi : \R^n \to \R^m$ be the projection into the common coordinate subspace,
    then $\supp(P_F) = (\pi(F \cap S_{i_1}), \dots, \pi(F \cap S_{i_m}))$.
    It is also assumed that the projection $\pi(F)$ is of dimension less than $m$.
    Then the supports of $P_F$ lie in an affine subspace of dimension less than $m$. 
    Consequently
    \begin{equation}
        \label{equ:mv-PF}
        \mvol(\supp(P_F)) = 
        \mvol(\pi(F \cap S_{i_1}), \dots, \pi(F \cap S_{i_m})) = 0.
    \end{equation}
    Notice that the coefficients of $P_F$ is a subset of the coefficients of
    the original system $P = (p_1,\dots,p_n)$.
    So there is a nonempty Zariski open set $C_F$ among the choices of 
    coefficients for $P$ for which $P_F$ is also in general position.
    
    Since $\conv(\tilde{S})$ has finitely many faces and in Zariski topology, 
    any two nonempty open set must intersect, without loss of generality, 
    we may assume the coefficients for $P$ is chosen so that $P_F$ is in 
    general position for all proper positive dimensional faces $F$. 
    Then by \eqref{equ:mv-PF},
    $\V^*(P_F) = \V^*_0(P_F) = \varnothing$.
    Recall that $P_F(x_{j_1},\dots,x_{j_m})$ is simply $\init_{\boldalpha}(P_I)$ 
    but ignoring the $n-m$ variables that do not actually appear.
    So for any $\boldx = (x_1,\dots,x_n) \in \V^*(\init_{\boldalpha}(P_I))$,
    we must have $(x_{j_1},\dots,x_{j_m}) \in \V^*(P_F) = \varnothing$.
    Therefore $\V^*(\init_{\boldalpha}(P_I)) = \varnothing$ which implies
    $\V^*(\init_{\boldalpha}(A \cdot P)) = \V^*(\init_{\boldalpha}(P_I)) = \varnothing$.
%
    \smallskip
    
    We can now conclude that under the assumptions (A),(B), and (C)
    for any nonzero vector $\boldalpha \in \R^n$, $\init_{\boldalpha} (A \cdot P)$ 
    has no zero in $(\C^*)^n$.
    Then by \name{Bernshtein}'s Second Theorem (Theorem \ref{thm:bernshtein-b})
    $|\V_0^*(A \cdot P)| = n! \vvol(\conv(\tilde{S}))$, and consequently
    \begin{equation*}
        \mvol(\conv S_1, \dots, \conv S_n) = 
        |\V_0^*(P)| = 
        |\V_0^*(A \cdot P)| = 
        n! \vvol(\conv(\tilde{S})).
    \end{equation*}
    
    As in the previous case, since both sides of this equality are homogeneous of degree $n$ in a uniform 
    positive scaling, this result directly extend to cases with $S_i \subset \Q^n$.
    \qed
\end{proof}

\revision{
The above proofs also produced a byproduct that is potentially useful in 
numerical methods for solving system of Laurent polynomials:
}

\begin{proposition}
    \revision{
    Let $P = (p_1,\dots,p_n)$ be a Laurent polynomial system
    in general position 
    that satisfies the conditions in Theorem~\ref{thm:mainthma}
    or Theorem~\ref{thm:mainthmb}.
    For a generic square complex matrix $A$,
    the system $A \cdot P$ is also in general position.
    }
\end{proposition}

\revision{
    This property is particularly important in the
    numerical homotopy continuation methods for solving Laurent polynomial systems
    and is explored in \S~\ref{sec:homotopy}.
}

\begin{remark}[Automatic verification during volume computation] \label{rmk:automatic}
	The conditions for the two theorems proved above involve how faces of 
	$\conv(\tilde{S}) = \conv(S_1 \cup \cdots S_n)$ intersect 
	or fail to intersect each $S_i$.
	These, in principle, can be verified automatically as by-products from the
    process of computing the volume of $\conv(\tilde{S})$ using a 
    \revision{
    certain kind of subdivision algorithm:
    If the polytope is represented as the convex hull of a set of points,}
	a particularly simple procedure for constructing a subdivision for
	the convex polytope $\conv(\tilde{S})$ starts with the enumeration of
	all its facets.
	Then the collection of all the pyramids formed by the facets and a fixed
	interior point will be a subdivision of the polytope.
	With this construction, since the set of all facets is already generated,
	the conditions of the above theorems can be checked easily.
\end{remark}

\section{Turning mixed volume into semi-mixed types} \label{sec:semimix}

With minor modifications, the above proofs generalize directly to cases where
union is taken over only a subset of the polytopes whereby transforming the 
mixed volume into semi-mixed types even when the Newton polytopes are \emph{all distinct}.
This transformation will be particularly beneficial when a subset of polytopes have
nearly identical (but still different) set of vertices
(e.g. the Tensor Eigenvalue Problem to be discussed in \S \ref{sec:tensor}.
In such cases, taking the union over these similar subset of polytopes will not only
simplify the geometric information but also allow one to use the much more efficient
algorithms for mixed volume computation of semi-mixed types
\cite{chen_solutions_2014,gao_mixed_2003}.

\begin{corollary}[Semi-mixed version of Theorem \ref{thm:mainthma}]
    \label{thm:mainthmb-semi}
    Given nonempty finite sets $S_{i,j}\linebreak[1]\subset \Q^n$ for $i = 1,\dots,m$ 
    and $j=1,\dots,k_i$ with $k_i \in \Z^+$ and $k_1 + \cdots + k_m = n$, 
    let $Q_{i,j} = \conv (S_{i,j})$, $\tilde{S}_i = \bigcup_{j=1}^{k_i} S_{i,j}$, 
    and $\tilde{Q}_i = \conv ( \tilde{S}_i )$.
    If for each $i$, every positive dimensional face of $\tilde{Q}_i$ that
    intersect $S_{i,j}$ for some $j$ on at least two points must 
    intersect all $S_{i,1},\dots,S_{i,k_i}$, then
    \begin{equation*}
    \mvol (Q_{1,1},\dots,Q_{m,k_m}) =
    \mvol (\; 
        \underbrace{\tilde{Q}_1,\dots,\tilde{Q}_1}_{k_1}\,,
        \;\dots,\;
        \underbrace{\tilde{Q}_m,\dots,\tilde{Q}_m}_{k_m}
        \;).
    \end{equation*}
\end{corollary}


\begin{proof}
    Let $P = (p_{ij})_{i=1,\dots,m,j=1\dots,k_i}$ be a system of Laurent 
    polynomials in $\boldx = (x_1,\dots,x_n)$ with
    $p_{ij}(\boldx) = \sum_{\bolda \in S_{ij}} c_{i,j,\bolda} \boldx^{\bolda}$.
    Also define $P_i = (p_{i1},\dots,p_{ik_i})$ for each $i=1,\dots,m$.
    We further assume the coefficients are chosen so that $P$ is in general position.
    By \name{Bernshtein}'s First Theorem (Theorem \ref{thm:bernshtein-a}),
    \begin{equation*}
        |\V_0^*(P)| = \mvol (Q_{1,1}, \dots, Q_{m,k_m}).
    \end{equation*}
    Consider the randomization $A \cdot P$ of $P$ induced by a nonsingular 
    $n \times n$ block matrix
    \[
        A := 
        \left[
        \begin{smallmatrix}
        	A_1 &     &        &  \\
        	    & A_2 &        &  \\
        	    &     & \ddots &  \\
        	    &     &        & A_m
        \end{smallmatrix}
        \right]
    \]
    where each $A_i = [a^{(i)}_{j,k}]$ is a nonsingular $k_i \times k_i$ matrix.
    Since $A$ is nonsingular, $P(\boldx) = \boldzero$ if and only if 
    $(A \cdot P)(\boldx) = A \cdot (P(\boldx)) = \boldzero$. 
    So $\V_0^*(P) = \V_0^*(A \cdot P)$.
    
    As in the proof of Theorem \ref{thm:mainthma}, we further assume the entries 
    of $A_i$ for $i=1,\dots,m$ are chosen so that there are no cancellations of 
    terms in $A \cdot P$. 
    Then it is easy to verify that $A \cdot P$ is a semi-mixed system in the sense 
    that among its supports each $\tilde{S}_i$ appear $k_i$ times. That is,
    \[
        \supp(A \cdot P) = 
        (\; 
            \underbrace{\tilde{S}_1,\dots,\tilde{S}_1}_{k_1}\,,
            \;\dots,\;
            \underbrace{\tilde{S}_m,\dots,\tilde{S}_m}_{k_m}
        \;).
    \]
    By \name{Bernshtein}'s First Theorem (Theorem \ref{thm:bernshtein-a}), 
    \begin{equation*}
        |\V_0^*(A \cdot P) |\le
        \mvol (\; 
        \underbrace{\tilde{Q}_1,\dots,\tilde{Q}_1}_{k_1}\,,
        \;\dots,\;
        \underbrace{\tilde{Q}_m,\dots,\tilde{Q}_m}_{k_m}
        \;).
    \end{equation*}
    We shall establish the equality by examining the initial systems of $A \cdot P$.
    For a nonzero vector $\boldalpha \in \R^n$, 
    let $F_i = (\tilde{Q}_i)_{\boldalpha}$ for $i=1,\dots,m$.
    We consider the following cases:
    
    First, if $F_i$ is a vertex for some $i$, then each Laurent polynomial 
    in $\init_{\boldalpha}(A_i \cdot P_i)$ would have only one term, 
    and hence $\V^*(\init_{\boldalpha}(A_i \cdot P_i)) = \varnothing$.
    But
    \begin{equation*}
        \init_{\boldalpha} (A \cdot P) =
        \begin{bmatrix}
            \init_{\boldalpha} (A_1 \cdot P_1) \\
            \vdots \\
            \init_{\boldalpha} (A_m \cdot P_m)
        \end{bmatrix},
    \end{equation*}
    So $\V^*(\init_{\boldalpha}(A \cdot P))$, being a subset of
    $\V^*(\init_{\boldalpha}(A_i \cdot P_i))$, must also be empty.
    
    Now suppose $F_1,\dots,F_m$ are all positive dimensional.
    We shall fix an $i \in \{1,\dots,m\}$, and let 
    $I_i = \{ j \in \{1,\dots,k_i\} \mid F_i \cap S_{i,j} \ne \varnothing \}$
    which must be nonempty.
    Then by Lemma \ref{lem:face},
    \begin{equation}
        \init_{\boldalpha} (A_i \cdot P_i) = 
        \begin{bmatrix}
            \displaystyle
            \sum_{j=1}^{k_i} a^{(i)}_{1,j} 
            \sum_{\bolda \in F_i \cap S_{i,j}} c_{i,j,\bolda} \boldx^{\bolda}
            \\
            \vdots \\
            \displaystyle
            \sum_{j=1}^{k_i} a^{(i)}_{k_i,j} 
            \sum_{\bolda \in F_i \cap S_{i,j}} c_{i,j,\bolda} \boldx^{\bolda}
        \end{bmatrix} =
        A_{I_i} 
        \begin{bmatrix}
            \displaystyle
            \sum_{\bolda \in (F \cap S_{i,j})} c_{i,j,\bolda} \boldx^{\bolda} 
        \end{bmatrix}_{j \in I_i}
    \end{equation}
    where $A_{I_i}$ is the matrix containing columns of $A_i$ indexed by $I_i$
    which is of rank $|I_i| \le k_i$.
    Its kernel must be $\{\boldzero\}$, 
    so $\init_{\boldalpha}(A_i \cdot P_i) = \boldzero$ if and only if
    $\sum_{\bolda \in (F \cap S_{i,j})} c_{i,j,\bolda} \boldx^{\bolda} = 0$
    for each $j \in I_i$.
    
    If $|F \cap S_{i,j}| = 1$ for all $j \in I_i$, then each of the above
    Laurent polynomials on the right hand side has only one term and hence 
    no zeros in $(\C^*)^n$.
    That is, $\V^*(\init_{\boldalpha}(A_i \cdot P_i)) = \varnothing$.
    Consequently $\V^*(\init_{\boldalpha}(A \cdot P))$, being a subset of
    $\V^*(\init_{\boldalpha}(A_i \cdot P_i))$ must also be empty.
    
    Now suppose for each $i = 1,\dots,m$, $|F_i \cap S_{i,j}| > 1$ for at least 
    one $j \in I_i$, then by assumption each $F_i$ must intersects each of the 
    supports $S_{i,1},\dots,S_{i,k_i}$. 
    So $I_i = \{1,\dots,k_i\}$ for each $i$ and hence
    $\init_{\boldalpha} (A \cdot P) = A \cdot \init_{\boldalpha}(P)$.
    Recall that $A$ is assumed to be nonsingular, so 
    $\init_{\boldalpha} (A \cdot P) = \boldzero$ if and only if
    $\init_{\boldalpha}(P) = \boldzero$.
    But $P$ is assumed to be in general position, so 
    $\V^*(\init_{\boldalpha}(A \cdot P)) = \V^*(\init_{\boldalpha}(P)) = \varnothing$.

    The above cases have shown that for all nonzero vector $\boldalpha \in \R^n$, 
    $\V^*(\init_{\boldalpha}(A \cdot P)) = \varnothing$.
    Then by \name{Bernshtein}'s Second Theorem (Theorem \ref{thm:bernshtein-b}),
    \[
        \mvol(Q_{1,1},\dots,Q_{m,k_m}) =
        |\V_0^* (P)| = 
        |\V_0^* (A \cdot P)| = 
        \mvol ( 
        \underbrace{\tilde{Q}_1,\dots,\tilde{Q}_1}_{k_1},
        \dots,
        \underbrace{\tilde{Q}_m,\dots,\tilde{Q}_m}_{k_m}
        ).
    \]
    Since both sides of the equation are homogeneous of degree 
    $k_1 + \cdots + k_m = n$ in any positive uniform scaling of all the polytopes,
    this result extends directly to cases where each $S_{i,j} \subset \Q^n$.
    \qed
\end{proof}


\section{Case studies} \label{sec:cases}

In this section, we apply the theorems proved above to concrete problems from
real-world applications to reduce the mixed volume computation into the 
unmixed cases (volume computation) or semi-mixed cases.

\subsection{Synchronization for coupled oscillators on cycle graphs}

The spontaneous synchronization in networks of interconnected oscillators 
is a ubiquitous phenomenon that has been discovered and studied in a wide range 
of scientific disciplines including physics, biology, chemistry, and engineering 
\cite{acebron_kuramoto_2005,dorfler_synchronization_2014}.
Here, as a case study, we focus on a classic model proposed by 
\name{Y. Kuramoto} \cite{kuramoto_self-entrainment_1975}.
For a network of $N = n+1$ oscillators, labeled as $0,\dots,n$, 
one basic model for describing the behavior of the oscillators is the 
system of differential equations
\begin{equation}
    \frac{d \theta_i}{dt} = 
    \omega_i - 
    \sum_{j=0}^{n} a_{i,j} \sin(\theta_{i}-\theta_{j})
    \quad \text{ for } i = 0,\dots,n
    \label{equ:kuramoto-ode}
\end{equation}
where each $\omega_i$ is the natural frequency of the $i$-th oscillator,  
each $a_{i,j}$ describes the coupling strength between the $i$-th and $j$-th 
oscillator (how strongly they influence each other),
and each $\theta_i = \theta_i(t)$ is the angle of the $i$-th oscillator.
This is a mathematical representation of the tug of war between 
the oscillators' tendency to oscillate in their own natural frequencies 
and the influence of their connected neighbors.
``Synchronization'' occurs when these two forces reach an equilibrium
for each of the oscillators:
A configuration $\boldsymbol{\theta} = (\theta_0,\dots,\theta_n) \in \R^N$ 
is said to be in \emph{synchronization}\footnote{%
    There are several related concepts of ``synchronization''
    in this context, which are listed in \cite{dorfler_synchronization_2014}.
    Here we only study a version of the so called \emph{frequency synchronization},
    a.k.a. \emph{frequency critical points}.
    In the general context such points are characterized by all
    $\frac{d\theta_i}{dt}$ converging to a common value (not necessarily zero).
    However, after switching to a rotational frame of reference, it is equivalent
    to requiring $\frac{d\theta_i}{dt} = 0$ for $i=0,\dots,n$.
} if $\frac{d\theta_i}{dt} = 0$ for $i=0,\dots,n$ at $\boldsymbol{\theta}$, 
i.e., 
\begin{equation}
    \omega_i - \sum_{j=0}^{n} a_{i,j} \sin(\theta_{i}-\theta_{j}) = 0
    \quad \text{ for } i = 0,\dots,n
    \label{equ:kuramoto-sin}
\end{equation}
which will be referred to as the \emph{synchronization system} in our discussion.
The behavior of the synchronization solutions is completely understood
in cases where the underlying network forms a tree \cite{dekker_synchronization_2013}.
Here we shall study the simplest cases beyond trees --- cycle graphs.
That is, we assume for each $0 < i < n$ the $i$-th node is directly connected 
to two nodes: $i+1$ and $i-1$, and node 0 is directly connected to node $n$.
With the notation $i^+ = (i + 1) \mod N$ and $i^- = (i - 1) \mod N$,
the above system can be written as
\begin{equation}
    \omega_i - \sum_{j \in \{i^+,i^-\}} a_{i,j} \sin(\theta_{i}-\theta_{j}) = 0
    \quad \text{ for } i = 0,\dots,n
    \label{equ:kuramoto-ring}
\end{equation}


Note that the solutions to \eqref{equ:kuramoto-sin} has an inherent degree 
of freedom in the sense that if $(\theta_0,\dots,\theta_n)$ is a solution,
so is $(\theta_0 + t,\dots,\theta_n + t)$ for any $t$.
It is therefore a common practice to fix $\theta_0 = 0$ 
and remove the first equation from the system, leaving us $n = N-1$ 
nonlinear equations in the $n$ angles $\theta_1,\dots,\theta_n$.
Using the transformation proposed in \cite{ChenMehta2017Network}, we shall turn 
the above system into a Laurent polynomial system to which we can apply 
Theorem \ref{thm:mainthmb}.
First, using the identity
$\sin(\theta_{i} - \theta_{j}) = 
\frac{1}{2\imag}(
e^{ \imag(\theta_{i} - \theta_{j})} - 
e^{-\imag(\theta_{i} - \theta_{j})})$
\eqref{equ:kuramoto-ring} can be transformed into
\begin{equation}
    \label{equ:exp}
    \omega_{i} - 
    \sum_{j \in \{i^+,i^-\}} 
    \frac{a_{i,j}}{2\imag} (
    e^{ \imag \theta_{i}} e^{-\imag \theta_{j}} - 
    e^{-\imag \theta_{i}} e^{ \imag \theta_{j}}
    ) = 0
    \quad \text{ for } i = 1,\dots,n.
\end{equation}
With the substitution $x_{i} := e^{\imag \theta_{i}}$ for $i = 1,\dots,n$,
we obtain the Laurent polynomial system
\begin{equation}
    \label{equ:kuramoto-laurent}
    \omega_{i} - \sum_{j \in \{i^+,i^-\}} a_{i,j}' 
    (x_i x_j^{-1} - x_j x_i^{-1}) = 0
    \quad \text{ for } i = 1,\dots,n
\end{equation}
where $a_{i,j}' = \frac{a_{i,j}}{2\imag}$ and $x_0 = 1$ is a constant.
Clearly, real solutions of \eqref{equ:kuramoto-ring} are represented by
$\C^*$-solutions of \eqref{equ:kuramoto-laurent}.
We are interested in computing the BKK bound of this Laurent polynomial system.
By applying Theorem \ref{thm:mainthmb}, this can be done via volume computation
which is potentially easier to compute.

\begin{proposition}
    \label{pro:sync}
    Let $S_1,\dots,S_n$ be the supports of \eqref{equ:kuramoto-laurent}, then
    \[
        \mvol(\conv(S_1),\dots,\conv(S_n)) = 
        n! \vvol_n (\conv(S_1 \cup \cdots \cup S_n)).
    \]
\end{proposition}

\begin{proof}
    It is easy to verify that for each $i = 1,\dots,n$,
    \[
    S_i = \{ 
    \boldzero, 
    \bolde_i - \bolde_{i^+}, \bolde_{i^+} - \bolde_i,
    \bolde_i - \bolde_{i^-}, \bolde_{i^-} - \bolde_i,
    \}
    \]
    where $\bolde_0 := \boldzero$ and $\bolde_1,\dots,\bolde_n$ are the
    standard basis vectors for $\R^n$.
    Let $F$ be a positive dimensional face of 
    $\tilde{Q} = \conv(S_1 \cup \cdots \cup S_n)$,
    $\boldalpha$ be its inner normal, and 
    $h = h_{\boldalpha}(\tilde{Q})$.
    \smallskip
    
    First, if $h = 0$ then $\boldzero \in F$ since 
    $\inner{\boldzero}{\boldalpha} = 0 = h$. But $\boldzero \in S_j$ for each 
    $j=1,\dots,n$. So $F$ intersects all the supports $S_1,\dots,S_n$.
    \smallskip
    
    Now assume $h \ne 0$. Since $\boldzero \in \tilde{Q}$ we must have 
    $h \le \inner{\boldzero}{\boldalpha} = 0$. 
    So $h < 0$ and $\boldzero \not\in F$.
    Fix an $i \in \{1,\dots,n\}$ such that $F \cap S_i \ne \varnothing$.
    If $\bolde_i - \bolde_{i^+} \in F$ then 
    $\inner{\bolde_i - \bolde_{i^+}}{\boldalpha} = h$.
    In that case,
    \[
    \inner{\bolde_{i^+} - \bolde_i}{\boldalpha} = 
    - \inner{\bolde_i - \bolde_{i^+}}{\boldalpha} = -h > h
    \]
    and hence $\bolde_{i^+} - \bolde_{i} \not\in F$. 
    By reversing the signs, we can also conclude that
    if $\bolde_{i^+} - \bolde_{i} \in F$ then $\bolde_i - \bolde{i^+} \not\in F$.
    By a similar argument, if $F$ contains $\bolde_i - \bolde_{i^-}$ 
    then $F$ cannot contain $\bolde_{i^-} - \bolde_i$ and vice versa.
    Consequently, unless $F \cap S_i$ is a singleton $F$ must contain
    one point from each of the subsets 
    $\{ \bolde_i - \bolde_{i^+}, \bolde_{i^+} - \bolde_i \}$ and
    $\{ \bolde_i - \bolde_{i^-}, \bolde_{i^-} - \bolde_i \}$.
    In other words, unless $F \cap S_i$ is a singleton,
    \begin{itemize}[leftmargin=4ex]
        \item $F$ intersects both $S_{i^-}$ and $S_{i^+}$ if $i \in \{2,\dots,n-1\}$;
        \item $F$ intersects $S_{i^-}$ if $i = n$; or
        \item $F$ intersects $S_{i^+}$ if $i = 1$.
    \end{itemize}
    Recall that $i^-$ and $i^+$ are the neighbors of node $i$ in the cycle graph.
    Thus the above observation propagate through the cycle, and we can conclude 
    that the positive dimensional face $F$ either intersect some $S_j$ at a 
    single point or intersect $S_1,\dots,S_n$.
    Then by Theorem \ref{thm:mainthmb}
    \[
        \mvol(\conv(S_1),\dots,\conv(S_n)) = 
        n! \vvol_n (\conv(S_1 \cup \cdots \cup S_n)).
    \]
    \qed
\end{proof}

With the above result, we can turn the root counting problem for the
synchronization system \eqref{equ:kuramoto-sin} into a volume computation problem:
$\conv(S_1 \cup \cdots \cup S_n)$ can be written as
\begin{equation*}
    \nabla_n := \conv(\; \{\boldzero\} \;\cup\;
    \{ \, 
    \bolda_i - \bolda_{i+1}, \,
    \bolda_{i+1} - \bolda_i, \,
    \bolda_i - \bolda_{i-1}, \,
    \bolda_{i-1} - \bolda_i \,
	\}_{i=1,\dots,n} \;)
\end{equation*}
where
\[
    \begin{cases}
    \bolda_i = \boldzero &\text{if } i = 0 \text{ or } i = n+1 \\
    \bolda_i = \bolde_i  &\text{if } i = 1,\dots,n.
    \end{cases}
\]
Here $\nabla_n$ can be considered as a convex polytope that encodes the 
connectivity information of the underlying network with each direct connection 
(edge) contributing a pair of points in its construction.
The root count of the synchronization system is therefore bounded by
the normalized volume of this polytope:
\begin{corollary}
    \label{cor:sync-AP}
    For a cycle graph of $N = n+1$ oscillators, the number of  
    isolated synchronization solutions of \eqref{equ:kuramoto-ring} 
    is less than or equal to $n! \vvol_n (\nabla_n)$.
\end{corollary}

\subsection{Noonburg's neural network model}

We now apply our results to a classical family of polynomial systems proposed 
by \name{V. W. Noonburg} \cite{noonburg_neural_1989} for modeling the behavior 
of neural networks.
Though the BKK bounds of this family have been completely understood by the 
analysis of \name{Y. Zhang} \cite{zhang_mixed_2008}, they are still widely used 
as standard benchmark problems for testing solvers for polynomial systems 
\cite{attardi_posso_1995}.
The applicability of the results established above is therefore still a
meaningful indication of their usefulness.

One classical approach is to consider a neural network as a network of 
interconnected cells in which activity levels at each cell are inhibited 
or excited by the activity of the other cells. 
Mathematically this is equivalent to an interacting set of populations with 
the densities of each population affected negatively or positively by its 
competition or cooperation with the other populations. 
This analogy allows the use of the Lotka-Volterra model in the study of 
neural networks. 
A key mathematical problem in this approach is the polynomial system
\begin{equation}
	\begin{aligned}
		\sum_{j=1}^n \Delta_{ij} x_1 x_j^2 - c x_1 + 1 &= 0 \\
		&\vdotswithin{=} \\
		\sum_{j=1}^n \Delta_{ij} x_n x_j^2 - c x_n + 1 &= 0 
	\end{aligned}
	\label{equ:noon}
\end{equation}
in the variables $x_1,\dots,x_n$, where $\{\Delta_{ij}\}$,
with $\Delta_{ij}=0$ for $i=j$ and $\Delta_{ij} = \pm 1$ otherwise,
encode the types of connections between cell $i$ and cell $j$
while the constant $c$ dictates the activity level of cells.
In the following, this system is simply referred to as the 
\name{Noonburg} system.
Since our discussion focuses only on the monomial structure of this polynomial
system and not the coefficients, we shall fix $\Delta_{ij} = 1$ for $i \ne j$
and $c = 1.1$ (a particular choice of the coefficients that appeared in
several studies).

With Theorem \ref{thm:mainthmb}, we shall show the BKK bound of \eqref{equ:noon} 
(which gives the generic number of complex solution of this system)
is always the normalized volume of the convex hull of the union of the
Newton polytopes of the above system.
Consequently, the BKK bound can be computed as the unmixed case 
(volume computation).

\begin{proposition}
    Let $S_1,\dots,S_n$ be the supports of the \name{Noonburg} system 
    \eqref{equ:noon}. Then
    \[
	    \mvol(\conv(S_1),\dots,\conv(S_n)) = 
        n! \vvol(\,\conv(S_1 \cup \cdots \cup S_n)\,).
    \]
\end{proposition}

\begin{proof}
    We can see that for $i=1,\dots,n$,
    \[
	    S_i = 
	    \{\bolde_i, \boldzero \} \cup 
	    \{ \bolde_i + 2 \bolde_j \}_{j=1,\dots,n,j \ne i}\,.
    \]
    Let $F$ be a positive dimensional face of 
    $\tilde{Q} := \conv(S_1 \cup \cdots \cup S_n)$,
    and let $\boldalpha = (\alpha_1,\dots,\alpha_n) \in \R^n$ be its inner normal, 
    then $F = \{ \boldx \in \tilde{Q} \mid \inner{\boldx}{\boldalpha}=h\}$
    where $h = h_{\boldalpha}(\tilde{Q})$.
    Since $\boldzero \in \tilde{Q}$, we must have 
    $h \le 0 = \inner{\boldzero}{\boldalpha}$.
    Fix any $i \in \{1,\dots,n\}$ and assume $|F \cap S_i| \ge 2$.
    \smallskip
    
    (Case 1) Suppose $\boldzero \in F$, 
    then $\varnothing \ne F \cap S_j \ni \boldzero$ for each
    $j = 1, \dots, n$.
    \smallskip
    
    (Case 2) Suppose $\bolde_i,\bolde_i + 2\bolde_j \in F$ for some 
    $j \in \{1,\dots,n\}$ and $j \ne i$. Then 
    \begin{align*}
	    h &= \inner{\bolde_i}{\boldalpha} = \alpha_i \\
	    h &= \inner{\bolde_i+2\bolde_j}{\boldalpha} = \alpha_i + 2 \alpha_j
    \end{align*}
    Therefore $\alpha_i = h$ and $\alpha_j = 0$.
    But $\bolde_j + 2\bolde_i \in S_j \subseteq \tilde{Q}$, so we must have
    \begin{equation*}
	    h \le 
	    \inner{\bolde_j + 2\bolde_i}{\boldalpha} =
	    \alpha_j + 2\alpha_i =
	    2h,
    \end{equation*}
    i.e., $h \ge 0$. Recall that $h \le 0$. So we must have $h=0$ and hence
    $\boldzero \in F$ since $\inner{\boldzero}{\boldalpha} = 0 = h$.
    By the argument in case 1, $F$ intersects each $S_j$ for $j=1,\dots,n$.
    \smallskip
    
    (Case 3) Finally, suppose $\boldzero, \bolde_i \not\in F$,
    then $F \cap S_i = \{ \bolde_i + 2 \bolde_j \}_{j \in J}$ for some set
    $J \subset \{1,\dots,n\} \setminus \{i\}$ with $|J| \ge 2$. That is,
    \begin{align*}
	    h &= \inner{\bolde_i + 2\bolde_j}{\boldalpha} = \alpha_i + 2 \alpha_j
	    &&\text{for all } j \in J \\
	    h &< \inner{\bolde_i + 2\bolde_k}{\boldalpha} = \alpha_i + 2 \alpha_k
	    &&\text{for all } k \not\in J \text{ and } k \ne i.
    \end{align*}
    So we can conclude that $\alpha_j$ for all $j \in J$ are the same.
    Let $\beta$ be this constant, then $\alpha_i = h - 2 \beta$ and 
    $\alpha_k > \beta$ for all $k \in \{1,\dots,n\} \setminus (\{i\} \cup J)$.
    Now fix a distinct pair of $j,j' \in J$, and we shall consider 
    $S_j \subseteq \tilde{S}$.
    Since $\bolde_j + 2\bolde_i, \bolde_j + 2\bolde_{j'} \in S_j$, we must have
    \begin{align*}
    h &\le \inner{\bolde_j + 2\bolde_i\,\,\,}{\boldalpha} = 
    \alpha_j + 2\alpha_i\, =
    \beta + 2h - 4\beta = 2h - 3\beta \\
    h &\le \inner{\bolde_j + 2\bolde_{j'}}{\boldalpha} = 
    \alpha_j + 2\alpha_{j'} = 
    \beta + 2 \beta = 3\beta
    \end{align*}
    which reduce to $h \le 3\beta \le h$. That is, $h = 3\beta$ and hence
    $\alpha_i = h - 2\beta = \beta = \alpha_j < 0$ for all $j \in J$ and 
    $\alpha_i < \alpha_k$ for any $k \in \{1,\dots,n\} \setminus (\{i\} \cup J)$.
    Therefore,
    \begin{align*}
    &
    \begin{aligned}
    \inner{\bolde_i + 2\bolde_j\,\,}{\boldalpha} &= 3\beta = h 
    \;\forall\, j \in J \\
    \inner{\bolde_j + 2\bolde_i\,\,}{\boldalpha} &= 3\beta = h
    \;\forall\, j \in J \\
    \inner{\bolde_j + 2\bolde_{j'}}{\boldalpha} &= 3\beta = h
    \;\forall\, j,j' \in J, j\ne j'
    \end{aligned}
    &&
    \begin{aligned}
    \inner{\bolde_i\, + 2\bolde_k\,\,}{\boldalpha} &> 3\beta 
    \;\forall\, k \not\in J \\
    \inner{\bolde_k\, + 2\bolde_j\,\,}{\boldalpha} &> 3\beta 
    \;\forall\, k \not\in J \text{ and } j \in J \\
    \inner{\bolde_k + 2\bolde_{k'}}{\boldalpha} &> 3\beta 
    \;\forall\, k,k' \not\in J, k\ne k'.
    \end{aligned}
    \end{align*}
    Consequently, 
    \begin{align*}
    F \cap S_k &= \varnothing
    \;\text{for any } k \in \{1,\dots,n\} \setminus (\{i\} \cup J)\\
    F \cap S_j &= \{ \bolde_j + 2 \bolde_i \} 
    \cup \{ \bolde_j + 2 \bolde_{j'} \mid j' \in J, j' \ne j \}
    \;\text{for any } j \in J \\
    F \cap S_i &= \{ \bolde_i + 2 \bolde_j \mid j \in J \}
    \end{align*}
    Since for each $j \in J$, $F \cap S_j$ contains only points that are
    linear combinations of $\bolde_i$ and $\{ \bolde_j \mid j \in J \}$,
    $F \cap S_j$ is contained in a common coordinate subspace of dimension
    $m := |J| + 1$ spanned by $\{\bolde_i\} \cup \{ \bolde_j \mid j \in J \}$.
    Let $\pi : \R^n \to \R^m$ be the projection to this coordinate subspace.
    Since a point in $\boldx = (x_1,\dots,x_m) \in \pi(F)$ still must satisfy 
    $\beta x_1 + \cdots + \beta x_m = 3 \beta$, $\pi(F)$ is of dimension less
    than $m$.
    \smallskip
    
    The above three cases have exhausted all possibilities for which $F$
    intersects some $S_i$ on at least two points.
    That is, each positive dimensional face $F$ of $\tilde{Q}$ satisfies 
    one of the conditions listed in Theorem \ref{thm:mainthmb}. Therefore,
    \[
	    \mvol(\conv(S)_1,\dots,\conv(S)_n) = 
	    n! \vvol(\conv(S_1 \cup \cdots S_n)).
    \]
    \qed
\end{proof}

\subsection{Algebraic load flow equations} \label{sec:powerflow}

In power engineering, ``load-flow study'' is a mathematical analysis of the 
flow of electric power in an network of connected devices (a power system)
which are of crucial importance in the design, operation, and control of 
power systems \cite{kundur_power_1994}.
The ``load flow equations'', a family of nonlinear systems of equations, 
are among the most important mathematical tools in these studies.
Though many variations of these equations have been proposed and studied,
the fruitful algebraic approach 
\cite{baillieul_critical_1984,baillieul_geometric_1982,li_random_1987} 
has been the focus of many recent studies
\cite{chen_network_2016,guo_determining_1990,marecek_power_2014,mehta_algebraic_2015,mehta_numerical_2014}.
Here, the mathematical abstraction of a power system is captured by a graph 
$G = (B,E)$ together with a complex matrix $Y = (Y_{ij})$ where
$B = \{0,1,\dots,|B|-1\}$ is the finite set of nodes representing the ``buses'', 
$E$ is the set of edges representing the transmission lines connecting buses, 
and the matrix $Y$, known as the \textit{nodal admittance matrix}, 
assigns a nonzero complex value $Y_{ij}$ to each edge $(i,j) \in E$
(with $Y_{ij} = Y_{ji} = 0$ if $(i,j) \notin E$).
As a convention, we further require all nodes to be connected with itself via a
``loop'' to reflect the nonzero diagonal entries $Y_{ii}$ known as the 
\emph{self-admittances}.
The main interest of load flow study are the complex valued voltage on each bus
denoted by $v_0, v_1,\dots,v_n,u_0,u_1,\dots,u_n$ where $n = |B|-1$.
Among them, $v_0$ and $u_0$ are fixed constant while the rest are considered
to be variables.
In this setup, the algebraic load flow equations is a system of $2n$ 
polynomial equations in $2n$ variables:
\begin{equation}
    P_{G,Y} (v_1,\dots,v_n,u_1,\dots,u_n) = 
    \begin{cases}
        \sum_{k=0}^{n} \bar{Y}_{1k}\, v_1 u_k - S_1   = 0\\
        \quad \vdots \\
        \sum_{k=0}^{n} \bar{Y}_{nk}   v_n u_k - S_n   = 0\\[1.5ex]
        \sum_{k=0}^{n} Y_{1k}\,\, u_1 v_k - \bar{S}_1 = 0 \\
        \quad \vdots \\
        \sum_{k=0}^{n} Y_{nk}\,   u_n v_k - \bar{S}_n = 0
    \end{cases}
    \label{equ:powerflow}
\end{equation}
where $S_1,\dots,S_n \in \C^*$ are constants representing constraints 
chosen for the purpose of load flow studies.
The root count of this system in $(\C^*)^{2n}$ has been widely studied
\cite{baillieul_critical_1984,baillieul_geometric_1982,chen_network_2016,li_random_1987,marecek_power_2014}.
Here, we shall apply Theorem \ref{thm:mainthmb} to show that the BKK bound of
\eqref{equ:powerflow} reduces to the unmixed case for any graph with
more than two nodes.
In particular, we shall prove the conjecture that the BKK bound of 
\eqref{equ:powerflow} is precisely the normalized volume of the 
``adjacency polytope'' of the graph $G$ \cite{ChenMehta2017Network}.
Using the notations $\bolde_0 = \boldzero$ and $(\bolde_i,\bolde_j) \in \R^{2n}$ 
for the concatenation of the two vectors $\bolde_i \in \R^n$ and $\bolde_j \in \R^n$,
the adjacency polytope of $G$ is defined to be
\begin{equation}
    \label{equ:powerflow-ap}
    \nabla_G := \conv
        (
            \{ \boldzero \} \cup 
            \{ (\bolde_i, \bolde_j) \}_{(i,j) \in E} 
        ).
\end{equation}
Note that the graph $G$ is undirected, so for each edge $(i,j) \in E$,
we must have $(j,i) \in E$ by definition.
Therefore each edge in $G$ contributes a line segment
(from $(\bolde_i,\bolde_j)$ to $(\bolde_j,\bolde_i)$) in the construction
of $\nabla_G$.
The following conjecture is suggested in \cite{ChenMehta2017Network}.
Here, we provide a simple proof using Theorem \ref{thm:mainthmb}:
\begin{proposition}
    \label{pro:powerflow}
    Let $S_1,S_1',S_2,,S_2',\dots,S_n,S_n'$ be the supports of \eqref{equ:powerflow},
    then
    \[
        \mvol (\conv(S_1),\conv(S_1'),\dots,\conv(S_n),\conv(S_n') ) =
        (2n)! \, \vvol_{2n} (\nabla_G).
    \]
\end{proposition}

\begin{proof}
    It is easy to verify that for $i,j = 1,\dots,n$
    \begin{align*}
        S_i  &= \{\boldzero\} \cup \{ (\bolde_i,\bolde_j) \}_{j \in \mathcal{N}(i)}
        &
        S_j' &= \{\boldzero\} \cup \{ (\bolde_i,\bolde_j) \}_{i \in \mathcal{N}(j)}
    \end{align*}
    where $\mathcal{N}(k)$ denotes the set of nodes neighboring $k$ in $G$.
    Let $\tilde{S} = S_1 \cup S_1' \cup \cdots \cup S_n \cup S_n'$.
    Then we can see that
    \[
        \nabla_G = 
        \conv(S_1 \cup S_1' \cup \cdots \cup S_n \cup S_n') =
        \conv(\tilde{S}).
    \]
    Therefore we simply have to verify that the conditions listed in
    Theorem \ref{thm:mainthmb} are satisfied.
    Let $F$ be a proper positive dimensional face of $\nabla_G$.
    If $\boldzero \in F$ then $F$ must intersect all supports 
    since $\boldzero$ is a common point of all the supports.
    
    Now suppose $\boldzero \not\in F$. 
    Let $K$ be the set of $(i,j) \in E$ such that $(\bolde_i,\bolde_j) \in F$.
    With $\pi_1 : \Z^2 \to \Z$ and $\pi_2 : \Z^2 \to \Z$ being the projections 
    onto first and second coordinates respectively, define
    $I = \pi_1(K) \setminus \{0\}$ and $J = \pi_2(K) \setminus \{0\}$.
    Note that the only supports that contains a point $(\bolde_i,\bolde_j)$
    are $S_i$ and $S_j'$, so $(\bolde_i,\bolde_j) \in F$ implies that 
    $F \cap S_i \ne \varnothing$ and $F \cap S_j' \ne \varnothing$.
    Moreover, the supports that intersect $F$ are precisely
    $S_i$ for $i \in I$ and $S_j'$ for $j \in J$.
    Hence the total number of supports that intersect $F$ is $|I| + |J|$.
    Also notice that $F \cap \tilde{S} = \{(\bolde_i,\bolde_j)\}_{(i,j) \in K}$
    is contained in the coordinate subspace spanned by
    $\{(\bolde_i,\boldzero)\}_{i \in I} \cup \{(\boldzero,\bolde_j)\}_{j \in J}$
    which is of dimension $|I| + |J|$.
    Finally, we can see that the restriction of inner normal vector of $F$
    on this coordinate subspace must be a non-constant linear functional.
    Therefore the face $F$ satisfies the condition (C) in Theorem \ref{thm:mainthmb}.
    
    Then by Theorem \ref{thm:mainthmb},
    \[
        \mvol(S_1,\dots,S_1',\dots,S_n,S_n') =
        (2n)! \, \vvol_{2n} (\conv(\tilde{S})) = 
        (2n)! \, \vvol_{2n} (\nabla_G).
    \]
    \qed
\end{proof}

Returning to the algebraic context of the load flow equations, the above
proposition shall be interpreted as follows:

\begin{corollary}
    Given a graph $G$, the number of isolated solutions of the induced
    algebraic load flow system \eqref{equ:powerflow} in $(\C^*)^{2n}$ is
    bounded by the normalized volume of the adjacency polytope $\nabla_G$.
\end{corollary}

\begin{remark}
    We should also note that the load flow system \eqref{equ:powerflow} and the
    synchronization system \eqref{equ:kuramoto-sin}, though originally proposed in
    very different contexts, are intimately related.
    Indeed, the synchronization system can be considered as a specialized 
    version of the algebraic load flow system \cite{dorfler_synchronization_2014}.
    Therefore, in that sense, this result generalizes Corollary \ref{cor:sync-AP}.
\end{remark}

\subsection{Tensor eigenvalue problem} \label{sec:tensor}

Given a vector space $V$ isomorphic to $\C^n$ (or $\R^n$), a multi-linear form
$F : (V^*)^m \to \C^n$ naturally give rise to a tensor $\mathcal{A}$ of order $m$ 
which can be encoded as an $m$-way array of dimensions $[n \times \cdots \times n]$
with respect to a fixed coordinate system.
An operation central to the tensor eigenvalue problem is a form of contraction
between a tensor $\mathcal{A}$ and a vector $\boldx = (x_1,\dots,x_n) \in \C^n$
given by
\begin{align*}
    \mathcal{A} \, \boldx^{m-1} &\in \C^n &
    &\text{where} &
    (\mathcal{A} \, \boldx^{m-1})_j &= \sum_{i_2=1}^n \cdots \sum_{i_m=1}^n 
    a_{j,i_2,\dots,i_m} x_{i_2} \cdots x_{i_m}.
\end{align*}
Clearly, each entry in $\mathcal{A} \boldx^{m-1}$ is a homogeneous polynomial 
in the variables $x_1,\dots,x_n$ of degree $m-1$.
Based on this contraction operation, several different notion of 
\textit{tensor eigenvalues/eigenvectors} have been proposed.
They can be defined by a family of algebraic equations:
Given a positive integer $m'$, an \emph{eigenpair} of $\mathcal{A}$ is a tuple 
$(\boldx,\lambda) \in \C^n \times \C$ with $\boldx \ne \boldzero$ such that
\begin{equation*}
    (\mathcal{A} \, \boldx^{m-1})_j = \lambda x_j^{m'-1}
    \quad\text{for } j=1,\dots,n.
\end{equation*}
This definition depends on the choice of $m'$, and researchers have studied
properties of eigenpairs for different choices of $m'$ 
\cite{qi_eigenvalues_2005}.
An eigenpair defined thusly has an inherent degree of freedom:
if $(\boldx,\lambda)$ is an eigenpair, then so is 
$(t \cdot \boldx, t^{m-m'} \lambda)$ for any $t \in \C$. 
Eigenpairs related by this relation are considered to be \emph{equivalent}.
From a computational point of view, it is convenient to pick certain 
representatives from each equivalent class.
Several different convention for picking representatives have been proposed 
\cite{qi_eigenvalues_2005}), following standard practices of 
Numerical Algebraic Geometry \cite{sommese_numerical_1996,sommese_numerical_2005}, 
the additional linear equation $\boldc^\top \boldx = c_0$ for some complex vector 
$\boldc \in \C^n$ has been adopted \cite{zhou_computing_2015} as a natural 
criterion for picking a representative of an eigenpair ---
it can be verified that for randomly chosen $\boldc \in \C^n$ and $c_0 \in \C$,
with probability one, there is precisely one point $(\boldx,\lambda)$ in each 
equivalent class that will satisfy $\boldc^\top \boldx = c_0$.
With this additional ``normalization condition'', an eigenpair of $\mathcal{A}$
is defined by
\begin{equation}
    \label{equ:eigenpair}
    \begin{aligned}
        (\mathcal{A} \, \boldx^{m-1})_j &= \lambda x_j^{m'-1}
        \quad\text{for } j=1,\dots,n, \\
        \boldc^\top \boldx &= c_0.
    \end{aligned}
\end{equation}
The upper bound on the root count of this system is established via the theory
of toric algebraic geometry \cite{cartwright_number_2013}.

At the same time, the much broader notion of \emph{generalized tensor 
eigenvalue problem} \cite{chang_eigenvalue_2009,ding_generalized_2015} 
has been developed:
Given an $m$-order tensor $\mathcal{A}$ and an $m'$-order tensor $\mathcal{B}$, 
a vector $(\boldx, \lambda) = (x_1,\dots,x_n,\lambda) \in \C^{n+1}$ with 
$\boldx \ne \boldzero$ is said to be a \emph{$\mathcal{B}$-eigenpair} of 
$\mathcal{A}$ if
\begin{equation}
    \label{equ:B-eigenpair}
    \begin{aligned}
        \mathcal{A} \boldx^{m-1} &= \lambda \mathcal{B} \boldx^{m'-1} \\
        \boldc^\top \boldx &= c_0.
    \end{aligned}
\end{equation}
This notion unifies several related tensor eigenvalue problems proposed previously.
Indeed, it can be verified that \eqref{equ:eigenpair} is a special case of
\eqref{equ:B-eigenpair} where $\mathcal{B}$ is chosen to be the ``identity'' tensor
with $\mathcal{B}_{i,\dots,i} = 1$ and zero elsewhere.
The upper bound for the number of distinct number of $\mathcal{B}$-eigenpair
was established \cite{chen_computing_2015,zhou_computing_2015} via the theory 
of BKK bound.\footnote{%
    Actually, the stronger \name{Li-Wang} extension \cite{li_bkk_1996} of the 
    BKK bound was used in this analysis.
    This extension produces an upper bound of the root count of a 
    polynomial system in $\C^n$ (rather than $(\C^*)^n$).
    \revision{
        Alternatively, the stable mixed cells method~\cite{huber_bernsteins_1997}
        could potentially produce even tighter root count bound in $\C^n$,
        though it is more difficult to compute.
    }
}

Here, using Corollary \ref{thm:mainthmb-semi}, we shall show that even though 
\eqref{equ:eigenpair} is a special case of \eqref{equ:B-eigenpair},
the two have the same BKK bound.
We start with \revision{a} simple example.

\begin{example}
    For $n = 2$, $m = 3$, and $m'=2$, equation \eqref{equ:eigenpair} for the 
    $[2 \times 2 \times 2]$ tensor $\mathcal{A} = [a_{i,j,k}]$ is given by
    \begin{equation}
        \label{equ:eigenexample}
        \begin{aligned}
            a_{1,1,1} x_{1} x_{1} + a_{1,1,2} x_{1} x_{2} + 
            a_{1,2,1} x_{2} x_{1} + a_{1,2,2} x_{2} x_{2} -
            \lambda x_1^{1} &= 0 \\
            a_{2,1,1} x_{1} x_{1} + a_{2,1,2} x_{1} x_{2} + 
            a_{2,2,1} x_{2} x_{1} + a_{2,2,2} x_{2} x_{2} -
            \lambda x_2^{1} &= 0 \\
            c_1 x_1 + c_2 x_2 + c_3 x_3 - c_0 &= 0.
        \end{aligned}
    \end{equation}
    The supports of the these equations are
    \begin{equation}
        \begin{aligned}
            S_1 &= \{ (2,0,0), (1,1,0), (0,2,0), (1,0,1) \} \\
            S_2 &= \{ (2,0,0), (1,1,0), (0,2,0), (0,1,1) \} \\
            S_3 &= \{ (1,0,0), (0,1,0), (0,0,1), (0,0,0) \}
        \end{aligned}
    \end{equation}
    Here, $S_1$ and $S_2$ are almost identical.
    Therefore, for computing the mixed volume, it is advantageous to consider
    $\tilde{S} := S_1 \cup S_2$ and compute instead the mixed volume
    $\mvol(\tilde{S},\tilde{S},S_3)$ which is of semi-mixed type.
    Indeed, there is only one point in $S_1$ that is not in $S_2$
    and vice versa.
    Therefore, for any positive dimensional proper face $F$ of $\tilde{S}$,
    if $F \cap S_1$ contains at least two points, then it must also
    intersect $S_2$.
    Similarly, if $F \cap S_2$ contains at least two points then it must also
    intersect $S_1$.
    By Corollary \ref{thm:mainthmb-semi},
    \[
        \mvol(\conv(S_1),\conv(S_2),\conv(S_3)) = 
        \mvol(\conv(\tilde{S}),\conv(\tilde{S}),\conv(S_3)).
    \]
    \smallskip
    
    On the other hand, with the tensor $\mathcal{B} = [b_{i,j}]$ of order 2,
    the generalized tensor eigenvalue problem \eqref{equ:B-eigenpair} becomes
    \begin{equation}
        \label{equ:B-eigenexample}
        \begin{aligned}
            a_{1,1,1} x_{1} x_{1} + a_{1,1,2} x_{1} x_{2} + 
            a_{1,2,1} x_{2} x_{1} + a_{1,2,2} x_{2} x_{2} -
            \lambda b_{1,1} x_1 - \lambda b_{1,2} x_2 &= 0 \\
            a_{2,1,1} x_{1} x_{1} + a_{2,1,2} x_{1} x_{2} + 
            a_{2,2,1} x_{2} x_{1} + a_{2,2,2} x_{2} x_{2} -
            \lambda b_{2,1} x_1 - \lambda b_{2,2} x_2 &= 0 \\
            c_1 x_1 + c_2 x_2 + c_3 x_3 - c_0 &= 0.
        \end{aligned}
    \end{equation}
    The supports of the these equations are
    \begin{equation}
        \begin{aligned}
            T_1 &= \{ (2,0,0), (1,1,0), (0,2,0), (1,0,1), (0,1,1) \} \\
            T_2 &= \{ (2,0,0), (1,1,0), (0,2,0), (1,0,1), (0,1,1) \} \\
            T_3 &= \{ (1,0,0), (0,1,0), (0,0,1), (0,0,0) \}
        \end{aligned}
    \end{equation}
    So the BKK bound for the generalized tensor eigenvalue problem 
    \eqref{equ:B-eigenexample} is given by the mixed volume 
    $\mvol(\conv(T_1),\conv(T_2),\conv(T_3))$.
    But we can see $T_1$ and $T_2$ are identical, and $T_1 = T_2 = \tilde{S}$, 
    so \eqref{equ:B-eigenexample} and \eqref{equ:eigenexample} have the exact
    same BKK bound.
\end{example}
\medskip

In this example, via a simple counting argument, we can apply Corollary 
\ref{thm:mainthmb-semi} and show that for $n=2$, $m=3$, and $m'=2$, the two 
tensor eigenpair formulation \eqref{equ:eigenpair} and \eqref{equ:B-eigenpair} 
have the exact same BKK bound.
Indeed, this result hold for any dimension/order:
That is, for a fixed $m'$, \eqref{equ:eigenpair} and \eqref{equ:B-eigenpair}
have the same BKK bound:

\begin{proposition}
    For fixed integers $n$, $m$, and $m'$, all greater than 2, 
    the two polynomial systems 
    \eqref{equ:eigenpair} and \eqref{equ:B-eigenpair} have the same BKK bound.
\end{proposition}

\begin{proof}
    Both polynomial systems have $n + 1$ equations in the $n+1$ unknowns
    $x_1,\dots,x_n$, and $\lambda$.
    Let $\bar{m} = m' - 1$, then the supports of \eqref{equ:eigenpair} are
    \begin{align*}
        S_1 &= \{ (\bar{m} \cdot \bolde_1,1) \} \cup \{ (\bolda,0) \in (\N_0)^{n+1} \;:\; |\bolda|_1 = m - 1 \} \\
        &\vdots \\
        S_n &= \{ (\bar{m} \cdot \bolde_n,1) \} \cup \{ (\bolda,0) \in (\N_0)^{n+1} \;:\; |\bolda|_1 = m - 1 \} \\
        S_{n+1} &= \{ (\bolde_1,0), \dots, (\bolde_n,0), \boldzero \}
    \end{align*}
    where $\N_0 = \N \cup \{0\}$ and for a vector $\bolda = (a_1,\dots,a_n)$
    we use the notation $|\bolda|_1 = a_1 + \cdots + a_n$.
    Let $\tilde{S} = S_1 \cup \cdots \cup S_n$ and $\tilde{Q} = \conv(\tilde{S})$.
    Here, the first $n$ supports are almost identical.
    Indeed $S_i \setminus S_j = \{ (\bar{m} \cdot \bolde_i,1) \}$ for any 
    $j \ne i$.
    This observation allows the application of Corollary \ref{thm:mainthmb-semi}:
    Given any positive dimensional face $F$ of $\tilde{Q}$.
    Suppose $F$ intersects $S_i$ for some $i \in \{1,\dots,n\}$ with $F \cap S_i$
    containing at least two points. Since 
    $|S_i \setminus S_j| = 1$ for any $j=1,\dots,n$ and $j \ne i$.
    $F \cap S_i$ cannot be contained inside $S_i \setminus S_j$.
    That is, $F$ must intersect each $S_i$ for $i=1,\dots,n$.
    Therefore, by Corollary \ref{thm:mainthmb-semi},
    \begin{equation}
        \label{equ:eigenmv}
        \mvol (\conv(S_1),\dots,\conv(S_{n+1})) =
        \mvol (\;\underbrace{\tilde{Q},\dots,\tilde{Q}}_n\;,\; \conv(S_{n+1})\,)
    \end{equation}
    which is the BKK bound for the system \eqref{equ:eigenpair}. \smallskip
    
    On the other hand, the supports of the system \eqref{equ:B-eigenpair} are
    \begin{align*}
        T_1 &= 
        \{ (\bolda,0      ) \in (\N_0)^{n+1} \;:\; |\bolda|_1 = m - 1 \} \cup 
        \{ (\bolda,1) \in (\N_0)^{n+1} \;:\; |\bolda|_1 = \bar{m} \} \\
        &\vdots \\
        T_n &= 
        \{ (\bolda,0      ) \in (\N_0)^{n+1} \;:\; |\bolda|_1 = m - 1 \} \cup 
        \{ (\bolda,1) \in (\N_0)^{n+1} \;:\; |\bolda|_1 = \bar{m} \} \\
        T_{n+1} &= \{ (\bolde_1,0), \dots, (\bolde_n,0),\boldzero \}.
    \end{align*}
    Here, the first $n$ supports are identical, so this system is naturally of 
    semi-mixed type.
    Let $T := T_1 = \cdots = T_n$. Clearly, $\tilde{S} \subseteq T$, so 
    $\tilde{Q} = \conv(\tilde{S}) \subseteq \conv(T)$.
    But notice that each point of the form $(\bolda,1)$ with 
    $|\bolda|_1 = \bar{m}$ can be written as a convex combination of the points 
    $(\bar{m} \cdot \bolde_1,1),\dots,(\bar{m} \cdot \bolde_n,1) \in \tilde{S}$.
    Therefore $T \subset \conv(\tilde{S}) = \tilde{Q}$.
    Consequently, the BKK bound for \eqref{equ:B-eigenpair} is
    \begin{align*}
        \mvol(\conv(T_1),\dots,\conv(T_n),\conv(T_{n+1})) &=
        \mvol(\underbrace{\conv(T),\dots,\conv(T)}_n, T_{n+1}) \\ &=
        \mvol(\;\underbrace{\tilde{Q},\dots,\tilde{Q}}_n \;,\; \conv(T_{n+1}) ) \\ &=
        \mvol(\;\underbrace{\tilde{Q},\dots,\tilde{Q}}_n \;,\; \conv(S_{n+1}) ).
    \end{align*}
    Comparing this equality with \eqref{equ:eigenmv}, we can conclude that
    the systems \eqref{equ:eigenpair} and \eqref{equ:B-eigenpair} have the same
    BKK bound.
    \qed
\end{proof}

\section{Accelerating mixed volume computation} \label{sec:results}

Via the theory of BKK bound and polyhedral homotopy, mixed volume computation 
became an important problem in computational algebraic geometry.
Previous sections established the conditions under which the equality
$\mvol(Q_1,\dots,Q_n) = n! \vvol_n(\conv(Q_1 \cup \cdots \cup Q_n))$ holds
and demonstrated its use in concrete problems from applications.
In this section we show the substantial computational advantage that one could
potentially achieve through this transformation.

The algebraic load flow system \eqref{equ:powerflow} is reviewed in
\S \ref{sec:powerflow}, and with Proposition \ref{pro:powerflow} we established that 
its BKK bound satisfies the condition given in Theorem \ref{thm:mainthmb} 
and can therefore be computed as the normalized volume of the 
``adjacency polytope'' \eqref{equ:powerflow-ap}.
In the following, we compared the actual CPU time\footnote{%
    Since most of the software packages to be used rely on randomized algorithms,
    the average of CPU time from 5 different runs are used in the table.
    All runs are performed on the same workstation equipped with an
    \tech{Intel\textsuperscript{\textregistered} Core\textsuperscript{TM} i5-3570K} 
    processor running at $3.4$GHz.
    For a meaningful comparison, \tech{Hom4PS-3}, which is designed to compute
    mixed volume in parallel, is configured to use only one thread (serial mode)
    in this case.
} consumed by various programs 
for computing the BKK bound of the algebraic load flow equations using 
these two different approaches (mixed volume v.s. normalized volume).
For mixed volume computation, we tested popular packages
\tech{DEMiCs} \cite{mizutani_demics:_2008,mizutani_dynamic_2007}, 
\tech{MixedVol-2.0} \cite{lee_mixed_2011}, 
and \tech{Hom4PS-3} \cite{chen_mixed_2014,chen_mixed_2017}.
For volume computation, we tested the widely used package \tech{lrs} 
\cite{avis_reverse_1993}
and a new package named \tech{libtropicana} \cite{chen_libtropicana:_2016}
(see \S \ref{sec:tropicana}) developed by the author specifically for this
project based on a pivoting algorithm similar to the core algorithm of \tech{lrs}.
Table \ref{tab:powerflow-ring} shows such comparisons for the algebraic load flow
equations induced by cycle graphs consisting of 14 to 18 nodes.
In all these cases, converting mixed volume computation into volume computation 
via Proposition \ref{pro:powerflow} brought consistently over 11 fold speedup.
\begin{table}[h]
    \centering
    \begin{tabular}{lrrrrrrrr}
    	\toprule
    	N.o. nodes                            &     14 &     15 &      16 &       17 &       18 &  \\ \midrule
    	BKK Bound                             &  57344 & 122880 &  262144 &   557056 &  1179648 &  \\
    	Mixed volume  via \tech{Hom4PS-3}     & 12.54s & 27.57s & 1m9.16s & 2m49.05s & 6m31.35s &  \\
    	Adj. polytope via \tech{libtropicana} &  1.02s &  2.35s &   5.68s &    13.55 &   31.66s &  \\
    	Speedup ratio                         &  12.29 &   11.7 &    12.2 &     12.5 &    12.36 &  \\ \bottomrule
    \end{tabular}
    \caption{
        Speedup ratio of the adjacency polytope method (using \tech{libtropicana}) 
        over the conventional mixed volume method (using \tech{Hom4PS-3}) for 
        computing the BKK bound of the algebraic load flow equations
        \eqref{equ:powerflow} induced by cycle graphs of sizes $N=14$ to $N=18$.
    }
    \label{tab:powerflow-ring}
\end{table}

\begin{wrapfigure}[13]{r}{0.40\textwidth}
    \centering
    \includegraphics[width=0.39\textwidth]{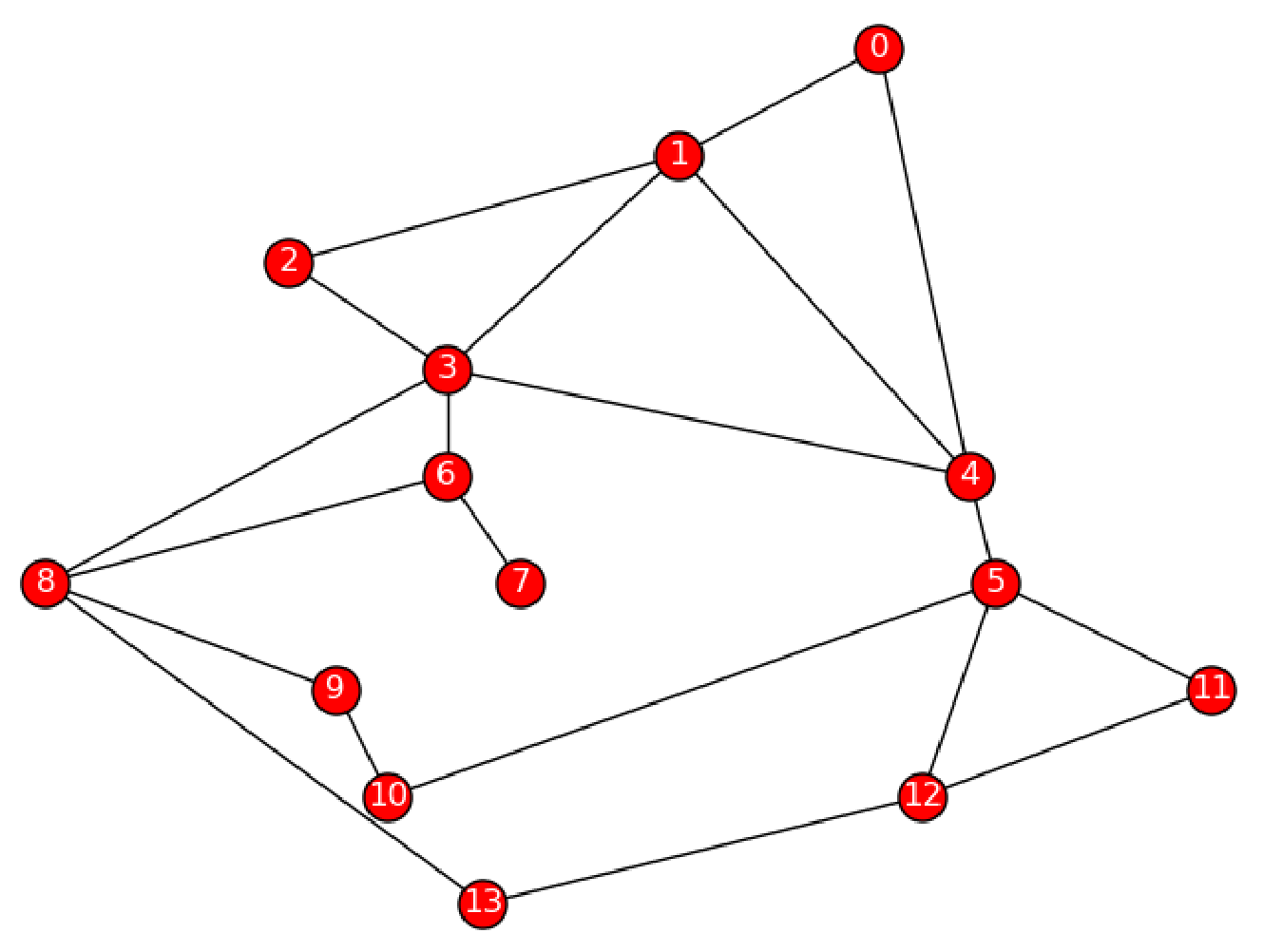}
    \caption{IEEE 14 bus system}
    \label{fig:ieee14}
\end{wrapfigure}
As another test case, we use the standard benchmark problem of the ``IEEE 14-bus'' 
system which represents a portion of the actual power grid of the Midwestern 
U.S.A. in the 1960s. 
It is one of the most widely used nontrivial test systems.
Consisting of 14 nodes (shown in Fig. \ref{fig:ieee14}), it induces an algebraic 
load flow system \eqref{equ:powerflow} of 26 Laurent polynomial equations in 
26 variables.
Consequently, the BKK bound computation, by definition, is the mixed volume of
26 polytopes in $\R^{26}$. Theorem \ref{thm:mainthmb} reduces this to a problem of
computing the normalized volume of a \emph{single} polytope in $\R^{26}$.
Table \ref{tab:ieee14} shows the CPU time consumed by the two approaches and
different programs.
In this case, converting mixed volume computation into volume computation 
via Proposition \ref{pro:powerflow} significantly accelerated the computation of the 
BKK bound (i.e. generic root count) for the algebraic load flow equations 
induced by the IEEE 14-bus system.
In particular, using \tech{libtropicana}, this transformation achieved 
over \emph{77 fold reduction in CPU time consumption}!

\begin{table}[h]
    \centering
    \setlength{\tabcolsep}{10pt} 
    \begin{tabular}{llrrr}
    	\toprule
    	Method                         & Program             & BKK Bound & CPU time &       Speedup \\ \midrule
    	\multirow{3}{*}{Mixed volume}  & \tech{Hom4PS-3}     &    427680 & 10m33.3s &            -- \\
    	                               & \tech{MixedVol-2.0} &    427680 & 12m11.5s &            -- \\
    	                               & \tech{DEMiCs}       &    427680 & 12m41.0s &            -- \\ \midrule
    	\multirow{2}{*}{Adj. polytope} & \tech{lrs}          &    427680 &    27.9s & $22.70\times$ \\
    	                               & \tech{libtropicana} &    427680 &     8.2s & $77.23\times$ \\ \bottomrule
    \end{tabular}
    \caption{
        CPU time consumed by various programs for computing the BKK bound of 
        the algebraic load flow system \eqref{equ:powerflow} induced by 
        the IEEE 14-bus system with two different basic approaches ---
        computing the mixed volume of 26 polytopes v.s. 
        computing the normalized volume of the adjacency polytope.
        The ``speedup'' column indicates the speedup ratio achieved by the
        adjacency polytope approach over the best run time of \tech{Hom4PS-3} 
        using the conventional approach of mixed volume computation.
    }
    \label{tab:ieee14}
\end{table}

\begin{remark} \label{rmk:mvol-vs-vol}
    From the view point of computational complexity, we must note that the
    problem of computing (exact) volume of a single polytope is \emph{not} 
    inherently easier than the problem of computing (exact) mixed volume of 
    several polytopes \cite{dyer_complexity_1998}.
    The computational advantage we intend to highlight here is the potentially 
    less complicated geometry of $\conv(Q_1 \cup \cdots \cup Q_n)$ comparing to 
    the $n$ polytopes $Q_1,\dots,Q_n$.
    For instance, with the algebraic load flow system induced by the IEEE 14-bus
    example, the 26 supports are originally defined by a total of 128 points 
    (with duplicates) while their union contain only 54 points.
\end{remark}

\section{Implications in polyhedral homotopy method} \label{sec:homotopy}


\revision{In addition to the problem of computing upper bounds for root counts,
the results established in this paper also has strong implications in
the polyhedral homotopy method for solving systems of Laurent polynomial systems.}

The problem of solving systems of nonlinear polynomial equations is a fundamental 
problem in mathematics that has a wide range of applications.
One important numerical approach to this problem is the
\emph{homotopy continuation methods} 
where a given ``target'' polynomial system to be solved is
continuously deformed into a closely related system that is trivial to solve. 
With an appropriate construction, the corresponding solutions also vary 
continuously under this deformation forming ``solution paths'' that connect the 
solutions of the trivial system to the desired solutions of the target system.
Then numerical ``continuation methods'' can be applied to track these 
paths and reach the target solutions.
Over the last few decades, these methods have been proven to be reliable, 
efficient, pleasantly parallel, and highly scalable.

Among a great variety of different homotopy constructions, the 
\emph{polyhedral homotopy method}, developed by \name{B. Huber} and \name{B. Sturmfels}
\cite{huber_polyhedral_1995}, is among the most efficient and flexible homotopy
constructions (together with ``regeneration'' based methods, e.g., 
\cite{hauenstein_regeneration_2011,hauenstein_regenerative_2011}).
We shall first briefly review the construction of polyhedral homotopy:
Given a system $P(\boldx) = (p_1,\dots,p_n)$ of $n$ Laurent polynomials 
in general position where $\boldx = (x_1,\dots,x_n)$ and
$p_i = \sum_{\bolda \in S_i} c_{i,\bolda} \boldx^{\bolda}$,
one is interested in finding $\boldx \in (\C^*)^n$ for which 
$P(\boldx) = \boldzero$.
For a choice of \emph{lifting functions} 
$\boldsymbol{\omega} = (\omega_1,\dots,\omega_n)$ 
with each $\omega_i : S_i \to \Q$ having sufficiently generic images,
the polyhedral homotopy for $P$ with respect to the ``liftings'' 
$\boldsymbol{\omega}$ is given by
\begin{equation}
    \label{equ:polyhedral}
	H(\boldx,t) =
	\begin{cases}
		\sum_{\bolda \in S_1} c_{1,\bolda} \boldx^{\bolda} t^{\omega_1(\bolda)} \\
        \hspace{8ex}\vdots\\
		\sum_{\bolda \in S_n} c_{n,\bolda} \boldx^{\bolda} t^{\omega_n(\bolda)} . \\
	\end{cases}
\end{equation}
Clearly, $H(\boldx,1) \equiv P(\boldx)$.
It can also be shown that for $P$ in general position, the solutions of
$H(\boldx,t) = \boldzero$ varies smoothly as $t$ varies in $(0,1) \subset \R$
forming smooth \emph{solution paths}.
Moreover, all isolated solutions of $P(\boldx) = \boldzero$ in $(\C^*)^n$
can be obtained as end points of these solution paths at $t=1$.
Since for any fixed $t \in (0,1)$, $H(\boldx,t)$ has the same supports
as $P$ itself, the number of solution paths is precisely the BKK bound
$\mvol(\conv(S_1),\dots,\conv(S_n))$.
It is intuitively clear that the total number of paths is a key factor
in the overall computational complexity of the polyhedral homotopy method.
An apparent difficulty in the above construction is that at $t=0$,
$H(\boldx,t) = H(\boldx,0) \equiv \boldzero$ (or becomes undefined)
and hence the starting points of the paths cannot be identified.
This difficulty is surmounted via a process known as ``mixed cell computation''
\cite{huber_polyhedral_1995}.
Once the mixed cells are computed, the starting points of the solution paths
can be located easily and efficiently.
Then numerical continuation methods can be applied to trace these paths
and reach all isolated solutions of $P(\boldx) = \boldzero$ in $(\C^*)^n$.
The results established above have important implications in the application
of polyhedral homotopy method:

\begin{proposition}
    Let $P(\boldx) = P(x_1,\dots,x_n)$ be a system of $n$ Laurent polynomials 
    in general position.
    Under the assumptions of Theorem \ref{thm:mainthma} or Theorem \ref{thm:mainthmb},
    for almost all $n \times n$ matrix $A$, the polyhedral homotopies 
    \eqref{equ:polyhedral} constructed for $P$ and $A \cdot P$ define the 
    same number of solution paths.
\end{proposition}

From a computational view point, this transformation from the problem of solving 
$P = \boldzero$ to the problem of solving its randomization $A \cdot P = \boldzero$
has the following potential benefits:
\begin{enumerate}
	\item 
		Turning the homotopy function $H$ into the ``unmixed'' form 
		where each equation involves the exact same set of terms
		significantly simplifies the scheme for simultaneous evaluation
		of $H$ and its partial derivatives 
        (e.g. \cite{kojima_efficient_2008,lee_hom4ps-2.0:_2008,verschelde_evaluating_2012})
        which is a particularly
		computationally intensive task in this method.
		
	\item
		The matrix $A$ can be chosen to improve the numerical condition of 
        the equation $H(\boldx,t) = \boldzero$ which plays a crucially 
        important role in the overall efficiency and stability of the 
        numerical homotopy methods
        \cite{allgower_numerical_1990,li_numerical_2003,sommese_numerical_2005}.
        
    \item
        As demonstrated in \S \ref{sec:results}, this transformation potentially
        accelerates the computation of BKK bound which is a crucial preprocessing 
        step that can be particularly time consuming for large systems.

	\item
		In this ``unmixed'' form, the collection of mixed cells required for
		locating the starting points of homotopy paths is equivalent to a
		simplicial subdivision of the Newton polytope 
		\cite{gao_mixed_2003,huber_polyhedral_1995,li_numerical_2003,verschelde_algorithm_1999}.
		Therefore if such a subdivision is used to compute the normalized volume,
		the starting points of homotopy paths can be located easily as by-products.
		
\end{enumerate}

\section{\revision{Concluding remarks}}\label{sec:conclusion}

\revision{
In this paper, we established sufficient conditions under which
the mixed volume of several convex polytopes is exactly the normalized volume
of the convex hull of their union.
Though originally motivated by geometric observations (\S~\ref{sec:motivation}),
our proofs are purely algebraic and relied on the theory of BKK bound.
We also generalized the result to semi-mixed volume
(mixed volume of semi mixed systems where polytopes may carry multiplicity)
which appears naturally in various counting problems
including the classical problem of counting Nash equilibria
\cite{Emiris2014,McKelvey1997,McLennan1997} in game theory.

We applied the resulting theory to a wide range of well known problems
in science and engineering including
Noonburg's neural network model, Kuramoto model for synchronization, 
load flow equations from electric engineering, 
and the tensor eigenvalue problem.
In all these cases, the root counting problem originally formulated as mixed volume
can be reduced to the problem of volume computation or much simplified
semi-mixed volume computation problem.
Even though asymptotically volume computation is not inherently easier than mixed volume computation
\cite{dyer_complexity_1998},
these transformations greatly reduced the total number of vertices and polytopes involved
in these cases,
and as a result we expect great reduction in time and memory requirements
when computing these root counts.
In the case of the load flow equations, our experiment with a standard test case problem
(IEEE 14 bus) shows a 77 fold reduction in CPU time!
These results have since found deeper applications to the study of Kuramoto model
\cite{ChenDavis2018Toric,Chen2018Counting}.

The results established here are closely related to the works by
Fr\'ed\'eric Bihan and Ivan Soprunov appeared around the same time~\cite{BihanSoprunov2017}
in which deeper analysis from a geometric view point were carried out.
On the algebraic side, it is reasonable to ask if our results can be
generalized to the theory of Newton-Okunkov bodies~\cite{KavehKhovanskii2012Newton} 
which are much more powerful generalization of Newton polytopes.
Preliminary studies produced some positive answers:
\begin{itemize}
    \item
        In the cases of Kuramoto equations induced by cycle graphs,
        not only the mixed volume of Newton polytopes
        but also the mixed volume of Newton-Okunkov bodies
        are reduced to volume~\cite{Chen2018Counting}.

    \item
        The techniques employed in this study is also capable of
        reducing the mixed volume of Newton-Okunkov bodies
        into mixed volume of Newton polytopes~\cite{Chen2018}
        under a similar condition.
\end{itemize}
The general situation, however, remains an open problem.
}


\appendix

%

\section{Monotonicity of mixed volume} \label{sec:mixvol}
The mixed volume $\mvol(Q_1,\dots,Q_n)$, as a function that takes $n$ convex 
polytopes, monotone in each of its arguments in the sense that if 
$Q_1' \subseteq Q_1$ then $\mvol(Q_1',Q_2,\dots,Q_n) \le \mvol(Q_1,Q_2,\dots,Q_n)$.
The same applies for all arguments.
Since $Q_i \subseteq \tilde{Q} := \conv(Q_1 \cup \dots \cup Q_n)$ 
for each $i=1,\dots,n$, the inequality
\[
    \mvol(Q_1,\dots,Q_n) \le
    \mvol(\tilde{Q},\dots,\tilde{Q}) = 
    n! \vvol(\tilde{Q})
\]
always hold regardless of the relative position of the polytopes.
The present contribution shows that the equality can hold even when
each $Q_i$ is strictly contained in $\tilde{Q}$.

\section{Modifications to polyhedral homotopies}

The apparent limitations of the construction of the polyhedral homotopy
\eqref{equ:polyhedral} are that 
the target system $P(\boldx)$ is assumed to be in general position, 
zeros in $\C^n \setminus (\C^*)^n$ may not be reached, and 
the numerical condition of the equation $H(\boldx,t) = \boldzero$ may be poor.
These limitations are surmounted by modifications proposed in subsequent studies 
\cite{huber_bernsteins_1997,kim_numerical_2004,lee_hom4ps-2.0:_2008,li_bkk_1996,rojas_toric_1999}.
A commonly used extension of \eqref{equ:polyhedral} with respect to the same
liftings and target system is given by
\begin{equation*}
    H(\boldx,t) =
    \begin{cases}
        \sum_{\bolda \in S_1} 
        [(1-e^t)c^*_{1,\bolda} + e^t c_{1,\bolda}] 
        (B \boldx)^{\bolda} e^{\omega_1(\bolda) t} + (1-e^t) \epsilon_1^*\\
        \vdotswithin{=}\\
        \sum_{\bolda \in S_n} 
        [(1-e^t)c^*_{n,\bolda} + e^t c_{n,\bolda}] 
        (B \boldx)^{\bolda} e^{\omega_n(\bolda) t} + (1-e^t) \epsilon_n^* . \\
    \end{cases}
\end{equation*}
where $c_{i,\bolda}$ and $\epsilon_i$ are generic complex numbers
and $B \boldx = (b_1 x_1,\dots,b_n x_n)$ with $b_i \in \R^+$ is chosen to
properly improve the numerical stability.
It can be shown that as $t$ varies from $-\infty$ to $0$, the solutions
of $H(\boldx,t) = \boldzero$ also vary continuously forming smooth
solution paths that collectively reach all isolated zeros of the target system
$P(\boldx)$ in $\C^n$.
This extension has been adopted in 
\tech{PHoM} \cite{gunji_phom_2004},
\tech{Hom4PS-2.0} \cite{lee_hom4ps-2.0:_2008},
and \tech{Hom4PS-3} \cite{chen_hom4ps-3:_2014}.
A variation of it can also be found in recent versions of \tech{PHCpack} 
\cite{verschelde_algorithm_1999}.

\section{\tech{libtropicana}} \label{sec:tropicana}
The software package \tech{libtropicana}\footnote{%
    \url{https://github.com/chentianran/libtropicana}} is developed by the author specifically
to carry out the experiments shown in \S \ref{sec:results}.
Given a convex polytope in $\Z^n$, it computes a \emph{regular subdivision}
and also produces the normalized volume of the polytope as a byproduct.
It is based on a pivoting algorithm similar to the core algorithm of \tech{lrs}
\cite{avis_pivoting_1992}.
But unlike \tech{lrs}, which puts a special emphasis on memory efficiency
and accuracy, \tech{libtropicana} focuses on speed 
(potentially at the expense of higher memory consumption)
and moderate sized polytopes.
It is written completely in \tech{C++} with optional interface for leveraging 
\tech{BLAS} and \tech{spBLAS} (Sparse \tech{BLAS}) routines.
\tech{libtropicana} is open source software. 
Users may freely distribute its source under the terms of the LGPL license.


\end{document}